\documentclass{amsart}
\usepackage{amsmath,amssymb,enumerate,graphicx,hyperref,mathrsfs,verbatim,xypic}

\newcommand{\Z}{{\mathbb{Z}}}
\newcommand{\Q}{{\mathbb{Q}}}
\newcommand{\R}{{\mathbb{R}}}
\newcommand{\C}{{\mathbb{C}}}
\renewcommand{\S}{{\mathbb{S}}}

\newcommand{\rk}{{\operatorname{rank}\ }}
\newcommand{\img}{{\operatorname{im}\ }}

\newtheorem{thm}{Theorem} 

\newtheorem{prop}[thm]{Proposition}
\newtheorem{lemma}[thm]{Lemma}

\newtheorem*{propNull}{Proposition}

\begin{document}

\title{Singularities and stable homotopy groups of spheres II.}


\author{Andr\'as Sz\H ucs}
\address{ELTE Analysis Department, 1117 Budapest, P\'azm\'any P\'eter s\'et\'any 1/C, Hungary}
\email{szucs@math.elte.hu}

\author{Tam\'as Terpai}
\address{Alfr\'ed R\'enyi Institute of Mathematics, 1053 Budapest, Re\'altanoda u. 13-15., Hungary}
\email{terpai@math.elte.hu}



\maketitle

\renewcommand{\thefootnote}{}
\footnotetext{{\it Acknowledgement:} The authors were supported by the National Research, Development and Innovation Office NKFIH (OTKA) Grants NK 112735 and K 120697 and partially supported by ERC Advanced Grant LDTBud.}

\section{Introduction and motivation}

Understanding generic smooth maps includes the following ingredients:

\begin{enumerate}[a)]
\item Local forms: they describe the map in neighbourhoods of points (see Whitney \cite{Whitney}, Mather \cite{Mather}, Arnold \cite{Arnold1}, \cite{Arnold2} and others).
\item Automorphism groups of local forms: they describe the maps in neighbourhoods of singularity strata (a stratum is a set of points with equal local forms) (see J\"anich \cite{Janich}, Wall \cite{Wall}, Rim\'anyi \cite{RRPhD}, Sz\H ucs \cite{SzucsLNM}, Rim\'anyi-Sz\H ucs \cite{RSz}).
\item Clutching maps of the strata: they describe how simpler strata are incident to more complicated ones.
\end{enumerate}

The present paper is devoted to a systematic investigation of ingredient $c)$ (for some special class of singular maps), initiated in the first part of this paper, \cite{partI}. While the ingredients $a)$ and $b)$ were well studied, there were hardly any results concerning $c)$.

The clutchings of strata will be described by some classes of the stable homotopy groups of spheres, $G = \underset{n \geq 0}{\oplus} \pi^{\bf s}(n)$. We shall associate an element of $G$ to a singular map (more precisely, to its highest singularity stratum) that will be trivial precisely when the second most complicated singularity stratum can be smoothed in some sense around the most complicated one. If this is the case then considering the next stratum (the third most complicated) we define again a (higher dimensional) element of $G$, and  this will vanish precisely when this (third by complexity) stratum can be smoothed around the
highest one. And so on.

The aim of the paper is computing these classes explicitly. The actual computation is taken from Mosher's paper \cite{Mosher}, which is purely homotopy theoretic, and is not dealing with any singularity or any smooth maps. Mosher computes the spectral sequence arising from the filtration of $\C P^{\infty}$  by the spaces $\C P^n$ in the (extraordinary) homology theory formed by the stable homotopy groups. We show that in a special case of singular maps the classifying spaces of such singular maps can be obtained from the complex projective spaces by some homotopy theoretical manipulation. This allows us to identify the clutching classes describing the incidence structure of the singularity strata with the differentials of Mosher's spectral sequence.

\section{Spectral sequences}\label{section:SS}

\subsection{The Mosher spectral sequence}\label{section:CPSS}

Let us consider the following spectral sequence: $\C P^\infty$ is filtered by the subspaces $\C P^n$. Consider the extraordinary homology theory formed by the stable homotopy groups $\pi^{\bf S}_*$. The filtration $\C P^0 \subset \C P^1 \subset \dots$ generates a spectral sequence in the theory $\pi^{\bf S}_*$. The starting page of this spectral sequence is
$$
E^{1}_{p,q}=\pi^{\bf S}_{p+q}(\S^{2p}) = \pi^{\bf S}(q-p),
$$
containing the stable homotopy groups of spheres. This spectral sequence was investigated by Mosher \cite{Mosher}. Our first result is that (rather surprisingly) this spectral sequence coincides with another one that arises from singularity theory; we describe it now.

\subsection{The singularity spectral sequence}\label{section:singSS} (\cite{nagycikk}) Let $X^0 \subset X^1 \subset X^2 \subset \dots \subset X$ be a filtration such that for any $i$ there is a fibration $X^i \to B_i$ with fibre $X^{i-1}$. Then there is a spectral sequence with $E^1$ page given by $E^1_{p,q}=\pi_{p+q}(B_p)$ that abuts to $\pi_*(X)$. (Indeed, in the usual construction of a spectral sequence in the homotopy groups one has $E^1_{p,q}=\pi_{p+q}(X^p,X^{p-1})$ and in the present case this group can be replaced by $\pi_{p+q}(B_p)$.)

\par

Such an iterated fibration arises in singularity theory in the following way. Let $k$ be a fixed positive integer and consider all the germs of codimension $k$ maps, i.e. all germs $f:(\R^c,0) \to (\R^{c+k},0)$ where $c$ is a non-negative integer. Two such germs will be considered to be equivalent if
\begin{itemize}
\item they are $\mathcal A$-equivalent (that is, one of the germs can be obtained from the other by composing and precomposing it with germs of diffeomorphisms), or
\item one of the germs is equivalent to the trivial unfolding (or suspension) of the other, that is, $f$ is equivalent to $f \times id_\R : (\R^c \times \R^1,0) \to (\R^{c+k} \times \R^1)$. 
\end{itemize}
An equivalence class will be called a \emph{singularity} or a \emph{singularity type} (note that germs of maximal rank, with $\rk df =c$, form an equivalence class and this class is also a ``singularity''). For any codimension $k$ smooth map of manifolds $f: M^n \to P^{n+k}$, any point $x\in M$ and any singularity $\eta$ we say that $x$ is an \emph{$\eta$-point} if the germ of $f$ at $x$ belongs to $\eta$. There is a natural partial order on the singularities: given two singularities $\eta$ and $\zeta$ we say that $\eta < \zeta$ if in any neighbourhood of a $\zeta$-point there is an $\eta$-point.

\par

Now let $\tau$ be a sequence of singularities $\eta_0$, $\eta_1$, $\eta_2$, $\dots$ such that for all $i$ in any sufficiently small neighbourhood of an $\eta_i$ point there can only be points of types $\eta_0$, $\dots$, $\eta_{i-1}$. We say that a smooth map $f: M^n \to P^{n+k}$ is a \emph{$\tau$-map} if at any point $x\in M$ the germ of $f$ at $x$ belongs to $\tau$. One can define the cobordism group of $\tau$-maps of $n$-manifolds into $\R^{n+k}$ (with the cobordism being a $\tau$-map of an $(n+1)$-manifold with boundary into $\R^{n+k} \times [0,1]$); this group is denoted by $Cob_\tau (n)$. In \cite{nagycikk} it was shown that there is a classifying space $X_\tau$ such that
$$
Cob_\tau(n) \cong \pi_{n+k}(X_\tau).
$$
Let $\tau_i$ denote the set $\tau_i = \left\{ \eta_0, \dots, \eta_i \right\}$ and denote by $X^i$ the classifying space $X_{\tau_i}$. It was shown in \cite{nagycikk} that $X^0 \subset X^1 \subset \dots$ is an iterated fibration, that is, there is a fibration $X^i \to B_i$ with fibre $X^{i-1}$. The base spaces $B_i$ have the following description (given in \cite{nagycikk}): to the singularity $\eta_i$
one can associate two vector bundles, the universal normal bundle $\xi_i$ of the stratum formed by $\eta_i$-points in the source manifold and the universal normal bundle $\tilde \xi_i$ of the image of this stratum in the target. Let $T\tilde \xi_i$ be the Thom space of the bundle $\tilde \xi_i$ and let $\Gamma T\tilde \xi_i$ be the space $\Omega^\infty S^\infty T\tilde\xi_i = \underset{q \to \infty}{\lim} \Omega^q S^q T\xi_i$. Then $B_i = \Gamma T\tilde \xi_i$. The obtained fibration $X_{i-1} \hookrightarrow X_i \to B_i$ is called the \emph{key fibration} of $\tau_i$-cobordisms (see \cite[Definition 109]{nagycikk}).

\par

Let us recall shortly the construction of the bundles $\xi_i$ and $\tilde\xi_i$. Let $\eta_i^{root} : (\R^{c_i},0) \to (\R^{c_i+k},0)$ be the root of the singularity $\eta_i$, i.e. a germ with an isolated $\eta_i$ point at the origin. Let $Aut_{\eta_i^{root}} < Diff(\R^{c_i},0) \times Diff(\R^{c_i+k},0)$ be the automorphism group of this germ, that is, the set of pairs $(\phi,\psi)$ of diffeomorphism germs of $(\R^{c_i},0)$ and $(\R^{c_i+k},0)$, respectively, such that $\eta_i^{root} = \phi \circ \eta_i^{root} \circ \psi^{-1}$. J\"anich \cite{Janich} and Wall \cite{Wall} showed that a maximal compact subgroup of this automorphism  group can be defined (and is unique up to conjugacy). Let $G_i$ denote this subgroup. It acts naturally on $(\R^{c_i},0)$ and on $(\R^{c_i+k},0)$ and these actions can be chosen to be linear (even orthogonal). We denote by $\lambda_i$ and $\tilde\lambda_i$ the corresponding representations of $G_i$ in $GL(c_i)$ and $GL(c_i+k)$ respectively. Now $\xi_i$ and $\tilde \xi_i$ are the vector bundles associated to the universal $G_i$-bundle via the representations $\lambda_i$ and $\tilde \lambda_i$, i.e. $\xi_i = EG_i \underset{\lambda_i}{\times} \R^{c_i}$ and $\tilde\xi_i = EG_i \underset{\tilde\lambda_i}{\times} \R^{c_i+k}$.

\par

Recall also that the space $\Gamma T\tilde \xi_i$ is the classifying space of immersions equipped with a pullback of their normal bundle from $\tilde\xi_i$ (such immersions will be called \emph{$\tilde\xi_i$-immersions}). That is, if we denote by $Imm^{\tilde\xi_i}(m)$ the cobordism group of immersions of $m$-manifolds into $\R^{m+c_i+k}$ with the normal bundle induced from $\tilde\xi_i$, the following well-known proposition holds:

\begin{prop}\label{prop:immGamma} \textnormal{(\cite{Wells}, \cite{EcclesGrant})}
$$Imm^{\tilde\xi_i}(m) \cong \pi_{m+c_i+k}(\Gamma T\tilde\xi_i) \cong \pi^{\bf S}_{m+c_i+k}(T\tilde\xi_i).$$
\end{prop}

Hence there is a spectral sequence in which the starting page is given by the cobordism groups of $\tilde\xi_i$-immersions:
$$
E^1_{p,q} = \pi^{\bf S}_{p+q}(T\tilde\xi_p) = Imm^{\tilde\xi_p}(p+q-c_i-k)
$$
and that abuts to the cobordism groups of $\tau$-maps. The differentials of this spectral sequence encode the clutching maps of the singularities that belong to $\tau$. For example, one may take an isolated $\eta_s$ singularity $g:(D^{c_s},\S^{c_s-1}) \to (D^{c_s+k},\S^{c_s+k-1})$; it would correspond to a generator $\iota_s$ of $\pi^{\bf S}_{c_s+k}(T\tilde\xi_s)$ (which is either $\Z$ or $\Z_2$). On the boundary $\partial g:\S^{c_s-1} \to \S^{c_s+k-1}$ of the map $g$, the $\eta_{s-1}$-points are the most complicated and therefore they form a $\tilde\xi_{s-1}$-immersion. The cobordism class of this immersion is the image $d^1(\iota_s)$ of $\iota_s$ under the differential $d^1$. Now assume that $d^1(\iota_s)=0$, that is, the $\eta_{s-1}$-points of $\partial g$ form a null-cobordant $\tilde\xi_{s-1}$-immersion. By \cite[Theorem 8]{nagycikk} and \cite{keyfibration}, any such cobordism can be extended to a $\tau_{s-1}$-cobordism that connects $\partial g$ with a $\tau_{s-2}$-map $\partial_2g : M^{c_s-1} \to \S^{c_s+k-1}$. The $\eta_{s-2}$-points of $\partial_2g$ form a $\tilde\xi_{s-2}$-immersion, and its cobordism class is the image $d^2(\iota_s)$ of $\iota_s$ under the differential $d^2$. If this image is $0$, then again we can change $\partial_2g$ by a cobordism to eliminate $\eta_{s-2}$-points, and so on.

\par

So far we have described two spectral sequences. We will show that in a special case the second one coincides with the first one, and this will allow us to study the clutchings of the singular strata.(Our proof of this equality of the two spectral sequences is independent of the paper \cite{nagycikk} to which we referred in this section.)

\section{Codimension $2$ immersions and their projections}\label{section:prim}

Let $\gamma \to \C P^\infty$ be the canonical complex line bundle and let $\gamma_r$ be its restriction to $\C P^r$. Let us consider an immersion $f : M^n \looparrowright \R^{n+2}$ of an oriented closed manifold. We call $f$ a \emph{$\gamma_r$-immersion} if its normal bundle is pulled back from the bundle $\gamma_r$. Let $Imm^{\gamma_r}(n)$ denote the cobordism group of $\gamma_r$-immersions of $n$-dimensional manifolds into $\R^{n+2}$. Analogously to Proposition \ref{prop:immGamma} we have $Imm^{\gamma_r}(n) \cong \pi^{\bf S}_{n+2}(\C P^{r+1})$.

\par

{\bf Definition:}
A smooth map $g: M^n \to \R^{n+k}$ is called a \emph{prim} map if
\begin{itemize}
\item $\dim \ker dg \leq 1$, and
\item the line bundle formed by the kernels of $dg$ over the set $\Sigma$ of singular points of $g$ is trivialized.
\end{itemize}

{\bf Remark:} Note that choosing any smooth function $h$ on the source manifold $M$ of $g$ such that the derivative of $h$ in the positive direction of the kernels $\ker dg$ is positive gives an immersion $f=(g,h):M \looparrowright \R^{n+k} \times \R^1$. So $g$ is the {\emph{pr}}ojection of the {\emph{im}}mersion $f$, motivating the name ``prim map''.

\par

The singularity types of germs for which the kernel of the differential is $1$-dimensional form an infinite sequence $\Sigma^{1,0}$ (fold), $\Sigma^{1,1,0}$ (cusp), $\Sigma^{1,1,1,0}$ etc. Let us denote by $\Sigma^{1_r}$ the symbol $\Sigma^{1,\dots,1,0}$ that contains $r$ digits $1$. We call a prim map \emph{$\Sigma^{1_r}$-prim} if it has no singularities of type $\Sigma^{1_s}$ for $s>r$. The cobordism group of prim $\Sigma^{1_r}$-maps of cooriented $n$-dimensional manifolds into an $(n+1)$-manifold $N$ will be denoted by $Prim\Sigma^{1_r}(N)$, and we will use $Prim\Sigma^{1_r}(n)$ to denote $Prim\Sigma^{1_r}(\S^{n+1})$.

\begin{thm}\label{thm:main}
$$Prim\Sigma^{1_r}(n) \cong \pi^{\bf S}_{n+2}(\C P^{r+1}).$$
\end{thm}

The theorem is related to the spectral sequence described in Section \ref{section:SS} in the following way. Set $k=1$ and $\eta_i = \Sigma^{1_i}$, so that $\tau_r = \left\{ \Sigma^{1_i} : i\leq r \right\}$. Denote by $X^r$ the classifying space $X_{Prim\_\tau_r}$ of prim $\tau_r$-maps. The main result, implying that the spectral sequences of Section \ref{section:SS} coincide, is the following:

\begin{lemma}\label{lemma:main}
The classifying space $X^r$ (whose homotopy groups give the cobordism groups of prim $\tau_r$-maps) is
$$
X^r \cong \Omega \Gamma \C P^{r+1}.
$$
\end{lemma}
 
{\bf Remarks:}
\begin{enumerate}
\item Recall that the functor $\Gamma$ turns cofibrations into fibrations. That is, if $(Y,B)$ is a pair of spaces such that $B \subset Y \to Y/B$ is a cofibration, then $\Gamma Y \to \Gamma (Y/B)$ is a fibration with fibre $\Gamma B$. The same is true if we apply to a cofibration the functor $\Omega\Gamma$, that is, $\Omega \Gamma Y \to \Omega \Gamma (Y/B)$ is a fibration with fibre $\Omega \Gamma B$.
\item Further let us recall the \emph{resolvent} of a fibration map. If $E \overset{F}{\to} B$ is a fibration then the sequence of maps $F\hookrightarrow E \to B$ can be extended to the left as follows. Turn the inclusion $F \hookrightarrow E$ into a fibration. Then it turns out that its fiber is $\Omega B$, the loop space of $B$. The fiber of the inclusion $\Omega B \to F$ is $\Omega E$, the fiber of $\Omega E \to \Omega B$ is $\Omega F$ etc. Now start with the filtration $\C P^0 \subset \dots \subset \C P^r \subset \C P^{r+1} \subset \dots \subset \C P^\infty$ considered by Mosher and apply $\Omega\Gamma$ to it; we get an iterated fibration filtration. So according to Subsection \ref{section:singSS} we obtain a spectral sequence that clearly coincides with Mosher's spectral sequence (obtained from the filtration $\C P^0 \subset \C P^1 \subset \dots \subset \C P^\infty$ by applying the extraordinary homology theory $\pi^{\bf S}_*$, of the stable homotopy groups). By Lemma \ref{lemma:main} this also coincides with the singularity spectral sequence described in Subsection \ref{section:singSS} for the special case of prim maps of oriented $n$-manifolds into $\R^{n+1}$.
\end{enumerate}

\par

{\bf Summary:} The singularity spectral sequence for codimension $1$ cooriented prim maps coincides with the spectral sequence of Mosher.

\par

{\bf Remark:} Theorem \ref{thm:main} follows obviously from Lemma \ref{lemma:main}.

\par

{\bf Remark:}  Note that the space of immersions of an oriented $n$-manifold $M$ into $\R^{n+2}$ (denote it by $Imm(M,\R^{n+2})$) is homotopy equivalent to the space of prim maps of $M$ into $\R^{n+1}$ (denoted by $Prim(M,\R^{n+1})$). Indeed, projecting an immersion $M^n \looparrowright \R^{n+2}$ into a hyperplane one gets a prim map $ M^n \to \R^{n+1}$ together with an orientation on the kernel of the differentials. Conversely, having a prim map $g: M^n \to \R^{n+1}$ one can obtain an immersion $f: M^n \looparrowright \R^{n+2}$ by choosing a function $h: M \to \R^1$ that has positive derivative in the positive direction of the kernels of $df$ and putting $f=(g,h)$. The set of admissible functions $h$ (for a fixed $g$) is clearly a convex set.

\par

Hence for any $M$ the projection induces a map $Imm(M, \R^{n+2}) \to Prim(M,\R^{n+1})$ that is a weak homotopy equivalence since the preimage of any point of $Prim(M,\R^{n+1})$ is a convex set. The space  $\Gamma \C P^\infty$ is the classifying space of codimension $2$ immersions. The previous remark implies that $\Omega \Gamma \C P^\infty$ is the classifying space of codimension $1$ prim maps. Moreover, the correspondence of the spaces of immersions and prim maps established above respects the natural filtrations of these spaces as the following lemma shows.

\begin{lemma}\label{lemma:filtration}
\begin{enumerate}[a)]
\item If an immersion $f: M^n \looparrowright \R^{n+2}$ has the property that its composition with the projection $\pi: \R^{n+2} \to \R^{n+1}$ has no $\Sigma^{1_i}$-points for $i>r$, then the normal bundle of $f$ can be induced from the canonical line bundle $\gamma_r$ over $\C P^r$.

\item The classifying space for the immersions into $\R^{n+2}$ that have $\Sigma^{1_r}$-map projections in $\R^{n+1}$  is weakly homotopically equivalent to the classifying space of immersions equipped with a pullback of their normal bundle from $\gamma_r$ (this latter being $\Gamma \C P^{r+1}$).
\end{enumerate}
\end{lemma}

{\bf Remark:} Lemma \ref{lemma:filtration} implies that $Prim\Sigma^{1_r}(n) \cong Imm^{\gamma_r}(n) \cong \pi^{\bf S}_{n+2}(T\gamma_r) = \pi^{\bf S}_{n+2}(\C P^{r+1})$, therefore Theorem \ref{thm:main} holds.

\begin{proof}
Let $f:M^n \looparrowright \R^{n+2}$ be an immersion such that the composition $g=\pi \circ f : M^n \to \R^{n+1}$ has no $\Sigma^{1_{r+1}}$-points. We show that the normal bundle $\nu_f$ of $f$ can be induced from the bundle $\gamma_r$. Let $\bf v$ denote the unit vector field $\partial/\partial x_{n+2}$ in $\R^{n+2}$; we consider it to be the upward directed vertical unit vector field. For any $x \in M$ consider the vector $\bf v$ at the point $f(x)$ and project it orthogonally to the normal space of the image of $f$ (i.e. to the orthogonal complement of the tangent space $df(T_xM)$). We obtain a section $s_0$ of the normal bundle $\nu_f \to M$ which vanishes precisely at the singular points of $g$. Denote by $\Sigma_1$ the zero-set of $s_0$, $\Sigma_1 = s_0^{-1}(0)$; then $\Sigma_1$ is a codimension $2$ smooth submanifold of $M$.
Let $\nu_1$ be the normal bundle of $\Sigma_1$ in $M$. Then $\nu_1$ is isomorphic to the restriction of $\nu_f$ to $\Sigma_1$ (the derivative of $s_0$ establishes an isomorphism). Consider now the vertical vector field $\bf v$ at the points of $\Sigma_1$ and project it orthogonally to the fibers of $\nu_1$. This defines a section $s_1$ of the bundle $\nu_1$, which vanishes precisely at the $\Sigma^{1,1}$-points of the map $g$ (at the singular set of the restriction of the projection $\pi$ to $\Sigma_1$). Denote this zero-set by $\Sigma_2$ and its normal bundle in $\Sigma_1$ by $\nu_2$. Next we consider $\bf v$ at the points of $\Sigma_2$ and project it to the fibers of $\nu_2$, obtaining a section $s_2$ of $\nu_2$ that vanishes precisely at the $\Sigma^{1_3}$-points of $g$, and so on. If $g$ has no $\Sigma^{1_{r+1}}$-points, then this process stops at step $r$ because $s_{r+1}^{-1}(0)$ will be empty.
\par
{\bf Claim:} In the situation above (when $s^{-1}_{r+1}(0)$ is empty) the bundle $\nu_f$ admits $(r+1)$ sections $\sigma_0, \dots, \sigma_r$ such that they have no common zero: $\cap_{i=0}^{r} \sigma_i^{-1}(0) = \emptyset$.
\par
{\it Proof of the claim:} Put $\sigma_0=s_0$. To define $\sigma_1$, we consider the section $s_1$ of the normal bundle $\nu_1$ of $\Sigma_1$ in $M$. This normal bundle $\nu_1$ is isomorphic to the restriction of $\nu_f$ to $\Sigma_1$, hence we can consider $s_1$ as a section of $\nu_f$ over $\Sigma_1$. Extend this section arbitrarily to a section of $\nu_f$ (over $M$); this will be $\sigma_1$. Similarly, $s_2$ is a section of the normal bundle $\nu_2$ of $\Sigma_2=s_1^{-1}(0)$ in $\Sigma_1$. The bundle $\nu_2$ is isomorphic to the normal bundle of $\Sigma_1$ in $M$ restricted to $\Sigma_2$, and this bundle in turn is isomorphic to $\nu_f$ restricted to $\Sigma_2$. Hence $s_2$ can be considered as a section of $\nu_f$ over $\Sigma_2$. Extending it arbitrarily to the rest of $M$ gives us $\sigma_2$. We continue the process and obtain $\sigma_3$, $\dots$, $\sigma_r$. Clearly these sections satisfy $\cap_{i=0}^{r} \sigma_i^{-1}(0) = \cap_{i=0}^{r} s_i^{-1}(0) = \emptyset$ as claimed.
\par
Next we define a map $\varphi: M \to \C P^r$ by the formula $\varphi(x) = [\sigma_0(x) : \dots : \sigma_r(x)]$, where we consider all $\sigma_j(x)$ as complex numbers in the fiber $(\nu_f)_x \cong \C$. This map is well-defined, since changing the trivialization of the fiber $(\nu_f)_x$ (while keeping its orientation and inner product) multiplies all the values $\sigma_j(x)$ by the same complex scalar, hence $\varphi(x) \in \C P^r$ remains the same.
\par
{\bf Claim:} $Hom_{\C}(\nu_f,\C) \cong \varphi^* \gamma_r$.
\par
Indeed, we can define a fibrewise isomorphism from $Hom_{\C}(\nu_f,\C)$ to $\gamma_r$ over the map $\varphi$ by sending a fiberwise $\C$-linear map $\alpha:\nu_f \to \C$ to $(\alpha(\sigma_0),\dots,\alpha(\sigma_r)) \in \C^{r+1}$. The image of this map over the point $x \in M$ lies on the line $\varphi(x)$ and since not all $\sigma_j$ vanish at $x$, the map is an isomorphism onto the corresponding fiber of $(\gamma_r)_{\varphi(x)}$.
\par
This proves part $a)$ of Lemma \ref{lemma:filtration}.
\par
The same argument can be repeated for any cooriented prim $\Sigma^{1_r}$-map of an $n$-dimensional manifold into an $(n+1)$-dimensional manifold. Thus for any target manifold $N$ we obtain a map $[N,X^r] = Prim\Sigma^{1_r}(N) \to [SN, \Gamma \C P^{r+1}] = [N,\Omega \Gamma \C P^{r+1}]$.
\par
Recall from Subsection \ref{section:singSS} that we can also assign to a $\Sigma^{1_r}$-map the set of its $\Sigma^{1_r}$-points, which form an immersion such that its normal bundle (in the target space) can be pulled back from the universal bundle $\tilde\xi_r$. This universal bundle $\tilde \xi_r$ in the case of prim Morin maps is trival (see \cite{partI}), hence $T\tilde\xi_r \cong \S^{2r+1}$. We will also use the existence of the key fibration mentioned in Subsection \ref{section:singSS} that says that this assignment fits into a fibration $X^{r-1} \subset X^r \to \Gamma \S^{2r+1}$.
\par
We now prove part $b)$ of Lemma \ref{lemma:filtration} by induction on $r$. For $r=0$ and any given target manifold $N$ the two classes of maps are $1$-framed codimension $2$ immersions into $N \times \R$ and codimension $1$ immersions into $N$ respectively. The spaces of these maps are weakly homotopically equivalent by the Compression Theorem \cite{compression}. The classifying space for both types of maps is $\Omega\Gamma\S^2$. For a general $r$, the operation of extracting the $\Sigma^{1_r}$-points from a prim map gives us the following commutative diagram:
$$
\xymatrix{
Prim\Sigma^{1_r}(N) = [N,X^r] \ar[r] \ar[d] & [N, \Gamma \S^{2r+1}]
\ar[d]_{=} \\
[SN,\Gamma \C P^{r+1}] = [N,\Omega\Gamma \C P^{r+1}] \ar[r] & [N, \Gamma \S^{2r+1}]
}
$$
The arrows of this diagram are natural transformations of functors, and thus correspond to maps of the involved classifying spaces (see \cite{Switzer}{Chapter 9, Theorem 9.13.}):
$$
\xymatrix{
X^r \ar[r] \ar[d] & \Gamma \S^{2r+1} \ar[d]_{=} \\
\Omega\Gamma \C P^{r+1} \ar[r] & \Gamma \S^{2r+1}
}
$$
By construction this diagram is commutative and the right-hand side map can be chosen to be the identity map. The horizontal arrows are Serre fibrations with fibers $X^{r-1}$ and $\Omega \Gamma\C P^{r}$ respectively. By the commutativity of the diagram fiber goes into fiber:
$$
\xymatrix{
X^{r-1} \ar@{}[r]|{\subset} \ar[d] & X^r \ar[r] \ar[d] & \Gamma \S^{2r+1} \ar[d]_{=} \\
\Omega\Gamma \C P^r \ar@{}[r]|{\subset} & \Omega\Gamma \C P^{r+1} \ar[r] & \Gamma \S^{2r+1}
}
$$
Consider the long homotopy exact sequence of the top and the bottom rows. The vertical maps of the diagram induce a map of these long exact sequences. By the induction hypothesis the map $X^{r-1} \to \Omega\Gamma \C P^r$ is a weak homotopy equivalence. The five-lemma then implies that the middle arrow is also a weak homotopy equivalence and $X^r$ is weakly homotopically equivalent to $\Omega\Gamma \C P^r$. This finishes the proof of Lemmas \ref{lemma:filtration} and \ref{lemma:main}.
\end{proof}

\section{Stable homotopy theory tools}\label{section:stable}
Mosher proved several statements concerning the differentials of his spectral sequence. Since it coincides with our singularity spectral sequence, Mosher's results translate into results on singularities. For proper understanding of the proofs of Mosher (which are highly compressed and not always complete) we elaborate them here. We start by recalling a few definitions and facts of stable homotopy theory.
\par
{\bf Definition:} Given finite $CW$-complexes $X$ and $Y$ we say that they are \emph{$S$-dual} if there exist iterated suspensions $\Sigma^p X$ and $\Sigma^q Y$ such that they can be embedded as disjoint subsets of $\S^N$ (for a suitable $N$) in such a way that both of them are deformation retracts of the complement of the other one.
\par
{\bf Definition:} Let $X$ be a finite cell complex with a single top $n$-dimensional cell. $X$ is \emph{reducible} if there is a map $\S^n \to X$ such that the composition $\S^n \to X \to X/sk_{n-1}X = \S^n$ has degree $1$.
\par
{\bf Definition:} Let $X$ be a finite cell complex with a single lowest (positive) dimensional cell, in dimension $n$. Then $X$ is \emph{coreducible} if there is a map $X \to \S^n$ such that its composition with the inclusion $\S^n \to X$ is a degree $1$ self-map of $\S^n$.
\par
Example: if $\varepsilon_B^n=B \times \R^n$ is the rank $n$ trivial bundle, then the Thom space $T\varepsilon^n_B$ is coreducible since the trivializing fibrewise map $\varepsilon^n_B \to \R^n$ extends to the Thom space and is identical on $\S^n$ (the one-point compactification of a fiber).
\par
{\bf Definition:} The space $X$ is $S$-reducible ($S$-coreducible) if it has an iterated suspension which is reducible (respectively, coreducible).
\par
\begin{prop}{\cite[Theorem 8.4]{Husemoller}}
If $X$ and $Y$ are $S$-dual cell complexes with top and bottom cells generating their respective homology groups, then $X$ is $S$-reducible if and only if $Y$ is $S$-coreducible.
\end{prop}
\par
\begin{lemma}\label{lemma:coreduciblebouquet}
If $X$ is a finite $S$-coreducible cell complex with $sk_n X = \S^n$, then $X$ is stably homotopically equivalent to the bouquet $\S^n \vee (X/sk_n X)$.
\end{lemma}
\par
\begin{proof}
The product of the retraction $r: X \to sk_n X$ and the projection $p: X \to X/sk_n X$ gives a map $f=(r,p):X \to \S^n \times (X/sk_n X)$. Replacing $X$ by its sufficiently high suspension we can assume that $\dim X <2n-1$. The inclusion $\S^n \vee (X/sk_n X) \to \S^n \times (X/sk_n X)$ is $2n-1$-connected, so in this case we can assume that $f$ maps $X$ into $\S^n \vee (X/sk_n X)$. Then $f$ induces isomorphisms of the homology groups, hence it is a homotopy equivalence.
\end{proof}

In what follows, we will encounter cell complexes such that after collapsing their unique lowest dimensional cell they become ($S$-)reducible or after omitting their unique top dimensional cell they become ($S$-)coreducible. In both of these cases, there is an associated (stable) $2$-cell complex:

{\bf Definition:} Let $X$ be a cell complex with a unique bottom cell $e^n$ such that $X/e^n$ is reducible, with the top cell $e^{n+d}$ splitting off. Then the attaching map of the cell $e^{n+d}$ to $X\setminus e^{n+d}$ can be deformed into $e^n$ -- one can lift the map $\S^{n+d} \to X/e^n$ in the definition of reducibility to a map $(D^{n+d},\partial D^{n+d}) \to (X,e^n)$. We define the \emph{$2$-cellification} of $X$ to be the $2$-cell complex $Z = D^{n+d} \underset{\alpha}{\cup} \S^n$ whose attaching map $\alpha$ is a deformation of the attaching map $\partial e^{n+d} \to X\setminus e^{n+d}$ into $e^n$.

{\bf Definition:} Let $X$ be a cell complex with a unique top cell $e^{n+d}$ such that $X\setminus e^{n+d}$ is coreducible, retracting onto the bottom cell $e^{n}$. We define the \emph{$2$-cellification} of $X$ to be the $2$-cell complex $Z = D^{n+d} \underset{\alpha}{\cup} \S^n$ whose attaching map $\alpha$ is the composition of the attaching map $\partial e^{n+d} \to X\setminus e^{n+d}$ with the retraction $X \setminus e^{n+d} \to e^n$.

Note that in both definitions there is a choice to be made: a deformation of the attaching map of the top cell into the bottom cell has to be chosen in the first definition, and a retraction onto the bottom cell has to be chosen in the second one. Making different choices can result in (even stably) different $2$-cell complexes, and any of them can be considered as possible $2$-cellifications. Next we show that taking the $S$-dual space to all possible $2$-cellifications of a complex $X$ we obtain precisely the $2$-cellifications of the $S$-dual complex $D_S[X]$.

\begin{lemma}\label{lemma:dual2cell}
Let $X$ and $Y$ be $S$-dual finite cell complexes with top and bottom homologies generated freely by the single top and bottom cells $e^{n+d}_X$, $e^n_X$, $e^{m+d}_Y$ and $e^m_Y$, respectively. Assume that in $X$, the cell $e^{n+d}_X$ is attached only to the bottom cell $e^n_X$ in a homotopically nontrivial way (in other words, $X/e^n_X$ is reducible). Then
\begin{itemize}
\item $Y \setminus e^{m+d}_Y$ admits a retraction $r$ onto $e^m_Y$ (that is, $Y \setminus e^{m+d}_Y$ is coreducible);
\item The set of stable homotopy classes of such retractions $r$ admits a bijection with the set of stable homotopy classes of maps $f : D^{n+d} \cup \S^n \to X$ that induce isomorphisms in dimensions $n$ and $n+d$ (here and later $D^p \cup \S^q$ denotes a $2$-cell complex); and
\item The $S$-duals of the $2$-cellifications of $X$ are precisely the $2$-cellifications of $Y$.
\end{itemize}
\end{lemma}

\begin{proof}
First we remark that one can identify retractions $r$ of $Y \setminus e^{m+d}_Y$ onto $e^m_Y$ with maps $Y \to D^{m+d} \cup \S^m$ that induce isomorphisms of homologies in dimensions $m$ and $m+d$. Indeed, any retraction $r$ induces identity on $H_m$, and extending the retraction by a degree $1$ map of the top cell gives us a map in homology that is an isomorphism in dimensions $m$ and $m+d$. On the other hand, any map from $Y$ to a $2$-cell complex that induces an isomorphism in homology in dimensions $m$ and $m+d$ is glued together from a map $r:Y \setminus e^{m+d}_Y \to \S^m$ and a map $h:(e^{m+d}_Y,\partial e^{m+d}_Y) \to (D^{m+d},\partial D^{m+d})$ that is homotopic to the identity relative to the boundary. The map $r$ induces an isomorphism on homology in dimension $m$, so it is homotopic to a map that maps (the closure of) $e^m_Y$ onto $\S^m$ homeomorphically. Then reparametrizing $\S^m$ we may assume this homeomorphism to be identical, so $r$ is a retraction.
\par
Returning to the proof of the lemma, since $e^{n+d}_X$ is attached homotopically nontrivially only to the bottom cell $e^n_X$, there exists a map $f:D^{n+d} \cup \S^n \to X$ that maps the two cells by degree $1$ maps onto the top and the bottom cell of $X$ respectively. The mapping induced in homology by $f$ is an isomorphism in dimensions $n+d$ and $n$, so its $S$-dual $D_S(f)$ induces isomorphisms in dimensions $m$ and $m+d$. Hence the restriction of the map $D_S(f)$ to $Y\setminus e_Y^{m+d}$ gives a retraction of $Y \setminus e^{m+d}_Y$ onto the closure of $e^m_Y$. Conversely, the $S$-duals of the possible retractions $r: Y \setminus e^{m+d}_Y \to \S^m$ are maps of the form $f: D^{n+d} \cup \S^n \to X$ that induce isomorphisms in homology in dimensions $n$ and $n+d$.
\par
For any such pair of $S$-dual maps $f$ and $r$ we have the cofibration
$$
D^{n+d} \cup \S^n \overset{f}{\to} X \to sk_{n+d-1} X/e^n_X
$$
and its $S$-dual cofibration
$$
 D^{m+d} \cup \S^m \overset{r \cup h}{\leftarrow} Y=D_S[X] \leftarrow D_S[sk_{n+d-1} X/e^n_X]\overset{\dag}{=}sk_{m+d-1} Y/e^m_Y
$$
(the equality $\dag$ being implied by Lemma \ref{lemma:restriction}). Since the middle and right-hand side terms are $S$-dual in these cofibration sequences, so are the left-hand side terms. That is, the possible $2$-cellifications of $X$ and $Y$ are $S$-dual.
\end{proof}

Later we will need some further technical lemmas.

\begin{lemma}\label{lemma:selfDual}
The $S$-dual of a $2$-cell complex is also a $2$-cell complex, with the same stable attaching map (up to sign).
\end{lemma}

\begin{proof}
Let $A$ be the $2$-cell complex in question, with cells in dimensions $n$ and $n+d$, and denote by $B'=D_S[A]$ its $S$-dual (sufficiently stabilized for all maps in the following argument to be in the stable range). Then there is a cofibration $\S^n \to A \to \S^{n+d}$, and $S$-duality takes it to a cofibration $\S^m \to B' \to \S^{m+d}$ (sufficiently stabilized for the maps of the proof to exist). The statement of the lemma is trivial if $d=1$ and the attaching map of $A$ is homotopic to a homeomorphism, so we assume that this is not the case, in particular, $A$ has nontrivial homology in dimension $n$. 
\par
We now construct a $2$-cell complex $B$ with cells in dimensions $m$ and $m+d$ as well as a map $f: B \to B'$ that will turn out to be a homotopy equivalence. Take a map $h:(D^{m+d}, \partial D^{m+d}) \to (B',\S^m)$ that represents a generator of $H_{m+d}(B'/\S^m) \cong H_{m+d}(\S^{m+d})$. We define $B$ as the $2$-cell complex obtained by attaching $D^{m+d}$ to $\S^m$ along $h|_{\partial D^{m+d}}$; the map $h$ extends to a map $f:B \to B'$. Comparing the long exact sequences in homology of the pairs $(B',\S^m)$ and $(B,\S^m)$, we see that $f$ induces isomorphisms of $H_*(\S^m)$ and $H_*(\S^{m+d})$ by construction, hence the $5$-lemma implies that $H_*(f) : H_*(B) \to H_*(B')$ is also an isomorphism in all dimensions. By Whitehead's theorem, $f$ must be a homotopy equivalence. Hence the $S$-dual of the $2$-cell complex $A$ is the $2$-cell complex $B$.
\par
It remains to show that the attaching map of $B$ is stably homotopic to that of $A$ (or its opposite, depending on the choice of orientations of the cells). Consider the long exact sequence of the stable homotopy groups of the cofibration $\S^n \to A \to \S^{n+d}$:
$$
\dots \to \pi^{\bf S}_{n+d}(A,\S^n) \cong \pi^{\bf S}_{n+d}(\S^{n+d}) \cong \Z \overset{\partial}{\to} \pi^{\bf S}_{n+d-1}(\S^n) \cong \pi^{\bf S}(d-1) \to \dots
$$
The boundary map $\partial$ takes the generator of $\Z$ to the attaching map of $A$. Taking $S$-duals we obtain the long exact sequence of the stable cohomotopy groups of the cofibration $\S^m \to B \to \S^{m+d}$:
$$
\dots \to \pi_s^m(\S^m) \overset{\delta}{\to} \pi_s^{m+1}(B,\S^m) \cong \pi_s^{m+1}(\S^{m+d}) \cong \pi^{\bf S}(d-1) \to \dots
$$
Since the $S$-duality establishes an isomorphism between these two sequences, the boundary map $\partial$ and the coboundary map $\delta$ coincide. It is hence enough to show that $\delta$ maps the generator of $\Z$ to the attaching map of $B$.
\par
In order to show this, consider the Puppe sequence of the cofibration $\S^m \to B \to \S^{m+d}$, which involves the coboundary map $\delta$:
$$
\S^m \to B \to B \cup Cone(\S^m) \sim \S^{m+d} \overset{\delta(id_{\S^{m+d}})}{\longrightarrow} Cone(B) \cup Cone(\S^m) \sim \S^{m+1} \to \dots
$$
Denote by $\beta : \S^{m+d-1} \to \S^m$ the attaching map of $B$. We claim that $\delta(id_{\S^{m+d}})$ is homotopic to the suspension $S\beta$. Indeed, denote by $g$ the following map of the sphere $\S^{m+d}$ into $B \cup Cone(\S^m)$: on the top hemisphere $g$ is a homeomorphism onto the top cell of $B$, and on the bottom hemisphere $g$ is the map $Cone(\beta): Cone(\S^{m+d-1}) \to Cone(\S^m)$. On the equator the two partial maps coincide and hence $g$ is a (homotopically) well-defined map. On the top cell of $B$ this map is $1$-to-$1$, so after identifying $B \cup Cone(\S^m)$ with $\S^{m+d}$ the map $g$ becomes a degree $1$ self-map of $\S^{m+d}$ and is hence a homotopy equivalence. Consequently the composition of $g$ with the collapse of $B$ in $B \cup Cone(\S^m)$ is homotopic to the map $\delta(id_{\S^{m+d}})$ from the Puppe sequence. But this composition map is the quotient of the bottom hemisphere map $Cone(\beta)$ after collapsing the boundary $\S^{m+d-1}$ in the source and the boundary $\S^m$ in the target, so it is the suspension $S\beta$, proving our claim.
\end{proof}

\par

\begin{lemma}\label{lemma:restriction}
Let $X$ and $Y$ be $S$-dual finite cell complexes, with the single top cell $e_X$ of $X$ generating the top homology of $X$ and the single bottom cell $e_Y$ of $Y$ generating the bottom homology $H_m(Y)$ of $Y$. Then $X \setminus int~ e_X$is $S$-dual to the quotient space $Y/\overline{e_Y}$. (Informally, omitting the top cell is $S$-dual to contracting the bottom cell.)
\end{lemma}

\begin{proof}
Let $i:X \setminus int~ e_X \to X$ be the inclusion. The map $i_*$ induced in homology by $i$ is an isomorphism in all dimensions except the top one, where it is $0$. Hence the $S$-dual of $i$ is a map $D_S[i]: Y \to D_S[X \setminus int~ e_X]$ that induces isomorphisms in all homology groups except the bottom one, where it is $0$. This means that the space $D_S[X \setminus int~ e_X]$ is $m$-connected and after a homotopy, we may assume that the map $D_S[i]$ maps $e_Y$ to a single point. Consequently, $D_S[i]$ factors through the contraction of the bottom cell $e_Y$ and yields a map $Y/\overline{e_Y} \to D_S[X \setminus int~ e_X]$ that induces an isomorphism in all homologies. Hence $D_S[X \setminus int~ e_X]$ is homotopy equivalent to $Y/\overline{e_Y}$, as claimed.
\end{proof}

\par

\begin{lemma}\label{lemma:multiplication}
Let $X$ be a finite cell complex with a single top cell in dimension $n$ that freely generates $H_n(X)$, and denote by $Y$ its $S$-dual that has a single bottom cell in dimension $m$ which freely generates $H_m(Y)$. Let $\tilde X$ denote the cell complex which has the same $n-1$-skeleton as $X$ and has a single $n$-cell whose attaching map is $q$ times the attaching map of the top cell in $X$. Similarly, let $\tilde Y$ be the quotient of $Y$ by a degree $q$ map of its bottom cell. Then $\tilde X$ and $\tilde Y$ are $S$-dual.
\end{lemma}

\begin{proof}
Denote by $D_S[\tilde X]$ the $S$-dual of $\tilde X$. Since there is a map $\phi: \tilde X \to X$ that has degree $q$ on the top cell and is a homeomorphism on the $n-1$-skeleton, its $S$-dual is a map $\psi: Y \to D_S[\tilde X]$ that induces an isomorphism on all homology groups except the $m$-dimensional one, where it is the multiplication by $q$. Since $D_S[\tilde X]$ can be chosen to be simply connected and has vanishing homology in dimensions $1$ to $m-1$, it is $m-1$-connected and without loss of generality we may assume that it does not contain cells of dimension $m-1$ or less. Similarly, we can assume that $D_S[\tilde X]$ has a single cell of dimension $m$, and the map $\psi$ sends the bottom cell of $Y$ into the bottom cell of $D_S[\tilde X]$. Since $\psi$ restricted to the bottom cell of $Y$ has to be a degree $q$ map, $\psi$ factorizes through the obvious map $r:Y \to \tilde Y$ and there exists a map $\hat \psi: \tilde Y \to D_S[\tilde X]$ such that $\psi = \hat \psi \circ r$. The map on homology induced by $\hat \psi$ is clearly an isomorphism in all dimensions above $m$, and since $r$ and $\psi$ both induce the multiplication by $q$ on $H_m$, $H_m(\hat \psi)$ is also an isomorphism. By Whitehead's theorem, $\hat \psi$ is a homotopy equivalence between $\tilde Y$ and $D_S[\tilde X]$, proving the claim of the lemma.
\end{proof}

\par

\section{Periodicity}\label{section:periodicity}

We define the Atiyah-Todd number $M_k$ to be the order of $\gamma_{k-1}$ in the group $J(\C P^{k-1})$ (see \cite{AtiyahTodd}, \cite{AdamsWalker}). It can be computed as follows: let $p^{\nu_p(r)}$ be the maximal power of the prime $p$ that divides the positive integer $r$. Then $M_k = \prod_{p \text{ prime}} p^{\nu_p(M_k)}$ and the exponents satisfy the formula
$$
\nu_p(M_k) = \max \left\{ r+\nu_p(r) : 1 \leq r \leq \left\lfloor \frac{k-1}{p-1} \right\rfloor \right\}.
$$
For $k \leq 6$ the numbers $M_k$ are the following:
\par
\begin{tabular}{c|c|c|c|c|c|c}
k & 1 & 2 & 3 & 4 & 5 & 6 \\
\hline
$M_k$ & 1 & 2 & $2^3 \cdot 3$ & $2^3 \cdot 3$ & $2^6\cdot 3^2\cdot 5$ & $2^6\cdot 3^2\cdot 5$
\end{tabular}
\par
\noindent
Note that $M_{k+1}$ is a multiple of $M_k$ for any $k$.
\par
Mosher proved that the spectral sequence of Subsection \ref{section:CPSS} is periodic with period $M_k$ in the following sense:

\begin{prop}\label{thm:periodicity}{\text {\cite[Proposition 4.4]{Mosher}}}
If $r \leq k-1$ then $E^r_{i,j} \cong E^r_{i+M_k,j+M_k}$. Moreover, the isomorphism $E^{k-1}_{i,j} \cong E^{k-1}_{i+M_k,j+M_k}$ commutes with the differential $d^{k-1}$.
\end{prop}

In particular, $d^1$ is $(2,2)$-periodic, $d^2$ and $d^3$ are $(24,24)$-periodic, etc. 

In the proof, we use the groups $J(X)$ that are defined as follows.

{\bf Definition:} For a topological space $X$, let $J(X)$ be the set of stable fiberwise homotopy equivalence classes of sphere bundles $S\xi$, where $\xi$ is a vector bundle over $X$. In particular, if for bundles $\xi$ and $\zeta$ the represented classes $[\xi]$ and $[\zeta]$ in $J(X)$ coincide, then the Thom spaces $T\xi$ and $T\zeta$ are stably homotopically equivalent. The natural surjective map $K(X) \to J(X)$ is compatible with the addition of vector bundles and hence transfers the abelian group structure of $K(X)$ onto $J(X)$ (see \cite{Husemoller}). Note that for any finite cell complex $X$ the group $J(X)$ is finite.

\begin{proof}
The definition of the first $(k-1)$ pages of the spectral sequence of the filtration $\C P^0 \subset \C P^1 \subset \dots \subset \C P^i \subset \dots$ is constructed using the relative stable homotopy groups $\pi^{\bf S}_*(\C P^m/\C P^l)$ (where $0\leq m-l \leq k$) and homomorphisms between these groups induced by inclusions between pairs $(\C P^m,\C P^l)$. We show that all these groups remain canonically isomorphic if we replace $\pi^{\bf S}_q(\C P^m,\C P^l)$ by $\pi^{\bf S}_{q+2M_k}(\C P^{m+M_k},\C P^{l+M_k})$.
\par
It is well-known that the quotient $\C P^m/\C P^l$ is the Thom space of the bundle $(l+1)\gamma_{m-l-1}$. Similarly $\C P^{m+M_k}/\C P^{l+M_k}$ is the Thom space of the bundle $(l+1+M_k)\gamma_{m-l-1}$. Since $M_k$ is the order of $J(\gamma_{k-1})$ in $J(\C P^{k-1})$, and that is obviously a multiple of the order of $J(\gamma_{m-l-1})$ in $J(\C P^{m-l-1})$ if $m-l-1 \leq k-1$, one has $J(M_k\gamma_{m-l-1})=0$. The homomorphism $\xi \mapsto J(\xi)$ is compatible with the sum of bundles, hence for any bundle $\xi$ over $\C P^{m-l-1}$ the Thom spaces $T(\xi \oplus M_k\gamma_{m-l-1})$ and $T(\xi \oplus M_k\varepsilon^1) = S^{2M_k}T\xi$ represent the same element in $J(\C P^{m-l-1})$ and hence are stably homotopically equivalent. Therefore the shift of indices ${(i,j) \mapsto (i+M_k,j+M_k)}$ maps the first $r$ pages of the spectral sequence into itself via the canonical isomorphism $\pi^{\bf S}_q(X) \cong \pi^{\bf S}_{q+2M_k}(S^{2M_k}X)$ for $r\leq k-1$.
\par
Let us give a more detailed overview of the isomorphism $E^r_{i,j} \cong E^r_{i+M_k,j+M_k}$ for $r \leq k-1$.
$$
E^1_{i,j}=\pi^{\bf S}_{i+j}(\C P^i/\C P^{i-1}) = \pi^{\bf S}_{i+j}(\S^{2i})
$$
$$
E^1_{i+M_k,j+M_k}=\pi^{\bf S}_{i+j+2M_k}(\C P^{i+M_k}/\C P^{i+M_k-1}) = \pi^{\bf S}_{i+j+2M_k}(\S^{2i+2M_k})
$$
hence $E^1_{i,j}$ is canonically isomorphic to $E^1_{i+M_k,j+M_k}$.
\par
The group $E^r_{p,q}$ is defined as the quotient
$$
E^r_{i,j} = \frac{Z^r_{i,j}}{B^r_{i,j}} = \frac{\img \left( \pi^{\bf S}_{i+j}(\C P^i,\C P^{i-r}) \to \pi^{\bf S}_{i+j}(\C P^i,\C P^{i-1})\right)}{\img \left( \pi^{\bf S}_{i+j+1}(\C P^{i+r-1},\C P^{i}) \overset{\partial}{\to} \pi^{\bf S}_{i+j}(\C P^i,\C P^{i-1})\right)}=
$$
$$
=\frac{\img \left( \pi^{\bf S}_{i+j}(T(i-r+1)\gamma_{r-1}) \to \pi^{\bf S}_{i+j}(\S^{2i})\right)}{\img \left( \pi^{\bf S}_{i+j+1}(T(i+1)\gamma_{r-2}) \to \pi^{\bf S}_{i+j}(\S^{2i})\right)}
$$
Now replacing $i$ by $i+M_k$ and $j$ by $j+M_k$ we obtain
$$
E^r_{i+M_k,j+M_k} = \frac{\img \left( \pi^{\bf S}_{i+j+2M_k}(T(i-r+1+M_k)\gamma_{r-1}) \to \pi^{\bf S}_{i+j+2M_k}(\S^{2i+2M_k})\right)}{\img \left( \pi^{\bf S}_{i+j+1+2M_k}(T(i+1+M_k)\gamma_{r-2}) \to \pi^{\bf S}_{i+j+2M_k}(\S^{2i+2M_k})\right)}
$$
But we can replace $T(a+M_k)\gamma_m$ by $S^{2M_k}Ta\gamma_m$ for any $a$ if $m<k$, hence we obtain a canonical isomorphism between $E^r_{i,j}$ and $E^r_{i+M_k,j+M_k}$. This proves the isomorphism of the groups, and the same argument goes through for the differentials $d^r$.
\end{proof}

\par

This proof relied on a sophisticated theorem of Adams, Atiyah and others that determined the order of $J(\gamma_{k-1})$. Next we give a simple independent proof of the fact that $d^1$ is $(2,2)$-periodic, elaborating \cite[Proposition 5.1]{Mosher}. Recall that $d^1_{s,s}$ is the map
$$
d^1_{s,s}: \pi_{2s}(\C P^s/\C P^{s-1}) \to \pi_{2s-1}(\C P^{s-1}/\C P^{s-2}).
$$
If $\iota_s \in \Z = \pi^{\bf S}_{2s}(\C P^s/\C P^{s-1})$ is the positive generator, then $d^1_{s,s}(\iota_s)$ is the (stable) homotopy class of the attaching map of the $2s$-cell of $\C P^s/\C P^{s-2}$ to the sphere $\C P^{s-1}/\C P^{s-2}$. Note that for any $s$ this map is trivial if and only if $\C P^s/\C P^{s-2}$ is stably homotopy equivalent to $\S^{2s} \vee \S^{2s-2}$.

\begin{lemma}\label{lemma:Sq2}
In the ring $H^*(\C P^s/\C P^{s-2})$ the cohomological operation $Sq^2$ is nontrivial if and only if $s$ is even.
\end{lemma}

{\bf Corollary:} For $s$ even the space $\C P^s/\C P^{s-2}$ is not stably homotopically equivalent to $\S^{2s} \vee \S^{2s-2}$, hence in this case $d^1_{s,s}(\iota_s)$ is not trivial.

\begin{proof}[Proof of Lemma \ref{lemma:Sq2}]
The projection $\C P^s \to \C P^s/\C P^{s-2}$ induces isomorphisms of the $\Z_2$-cohomology groups in dimension $2s$ and $2s-2$. Let us denote by $y$ the generator of the ring $H^*(\C P^s; \Z_2)$. Then $Sq^2 y^{s-1} = (s-1) y^s \neq 0$ if $s$ is even. The commutativity of the following diagram finishes the proof:
$$
\xymatrix{
H^{2s-2}(\C P^s/\C P^{s-2}) \ar[r]^{Sq^2} & H^{2s}(\C P^s/\C P^{s-2}) \\
H^{2s-2}(\C P^s) \ar[u]_{\cong} \ar[r]^{Sq^2} & H^{2s}(\C P^s) \ar[u]_{\cong}
}
$$
\end{proof}

We thus know that for $s$ even $d^1_{s,s}(\iota_s) \neq 0$, and it remains to show that under the same conditions $d^1_{s+1,s+1}(\iota_{s+1})=0$. We have established in Part I \cite{partI} that the differentials commute with the composition product. In particular, consider the map $d^1_{s,s+1}:E^1_{s,s+1} \to E^1_{s-1,s+1}$. Its domain is $E^1_{s,s+1} = \pi^{\bf S}_{2s+1}(\C P^s/\C P^{s-1}) \cong \pi^{\bf S}(1) = \Z_2$, with the generator traditionally denoted by $\eta$ (see \cite{Toda}). The codomain of the map is $\pi^{\bf S}(2) = \Z_2$ and is generated by $\eta \circ \eta$. This implies that $d^1_{s,s+1}(\eta)=\eta \circ d^1_{s,s}(\iota_s) = \eta \circ \eta \neq 0$. But then $d^1_{s+1,s+1}$ must vanish as it maps into the trivial kernel of $d^1_{s,s+1}$.

\subsection{Periodicity modulo $p$}

In \cite{Mosher}, it was stated without proof that a similar result holds when one considers the spectral sequence only from the point of view of $p$-components of the involved groups. We present a proof below. To simplify notation, we will use the convention $n_p = p^{\nu_p(n)}$ to denote the $p$-component of a natural number $n$.
\begin{propNull}{\cite[in the text]{Mosher}}
If $r \leq k-1$ then $E^r_{i,j}$ and $E^r_{i+(M_k)_p,j+(M_k)_p}$ are isomorphic modulo the class of groups of order coprime to $p$. Moreover, the $p$-isomorphism $E^k_{i,j} \cong E^k_{i+(M_k)_p,j+(M_k)_p}$ commutes with the differential $d^k$.
\end{propNull}

\par
For reader's convenience we recall the notion of $p$-equivalence that shall be used in the proof.
\par
{\bf Definition:} A map $f: X \to Y$ of simply connected spaces is a $p$-equivalence if it induces isomorphisms on the $p$-components of all homotopy groups. The $p$-equivalence of spaces is the finest equivalence relation according to which any two spaces that admit a $p$-equivalence map between them are equivalent.
\par
{\bf Definition:} Equivalently, simply connected spaces $X$ and $Y$ are $p$-equivalent if their $p$-localizations are homotopy equivalent.
\par
In order to imitate the previous proof we define the groups $J_p$:

{\bf Definition:} For a topological space $X$, let $J_p(X)$ be the set of stable fiberwise homotopy $p$-equivalence classes of sphere bundles $S\xi$, where $\xi$ is a vector bundle over $X$. In particular, if the classes $[\xi]=[\zeta]\in J_p(X)$, then the Thom spaces $T\xi$ and $T\zeta$ are stably homotopically $p$-equivalent.

\par

Note that the natural map $J_p: K(X) \to J_p(X)$ factors through $J: K(X) \to J(X)$ and transfers the abelian group structure of $K(X)$ and $J(X)$ to $J_p(X)$. Indeed, the Cartan sum operation in $K(X)$ corresponds to taking the fibrewise join in $J(X)$ and $J_p(X)$, and inverse elements exist in $J_p(X)$ since they can be found in $J(X)$.

\par

We are tempted to claim that $J_p(X)$ is actually just the $p$-component of $J(X)$ (which we denote by $J_p^{alg}(X)$ and call the ``algebraic'' $J_p$ to distinguish it from the ``geometric'' $J_p$ defined above), but we can only show this for the spaces $X=\C P^r$; this is however enough to prove Mosher's claim.

\begin{lemma}\label{lemma:JpCP}
$J_p(\C P^n) \cong J_p^{alg}(\C P^n)$. In other words, for any two stable vector bundles $\xi$ and $\eta$ over $\C P^n$ the sphere bundles $S\xi$ and $S\eta$ are stably fiberwise $p$-equivalent if and only if the class $[\xi-\eta] \in J(\C P^n)$ has order coprime to $p$. 
\end{lemma}

\begin{proof}[Proof of Lemma \ref{lemma:JpCP}]
Denote by $x = \gamma_n-1 \in K(\C P^n)$ the rank $0$ representative of the stable class of the tautological bundle. It is known (\cite[Theorem 7.2.]{AdamsK}) that $K(\C P^n) \cong \Z[x]/(x^{n+1})$.
\par

For any positive integer $q$ coprime to $p$ we have $J_p(\gamma_n) = J_p(\gamma_n^{\otimes q})$ since there is a fibrewise degree $q$ map from $\gamma_n$ to $\gamma_n^{\otimes q}$. Rewriting this in terms of $x$ and using $x^{n+1}=0$, we get that for any such $q$ we have
\begin{equation}\label{eq:Jmain}
J_p(1+x)=J_p((1+x)^q)=J_p\left(1+qx+{q \choose 2}x^2+\dots+{q \choose n}x^n\right).
\end{equation}
We show that we can choose values $q_1$, $q_2$, $\dots$, $q_m$ and $\lambda_1$, $\lambda_2$, $\dots$, $\lambda_m$ in such a way that considering the equality \eqref{eq:Jmain} for $q=q_1$, $\dots$, $q_m$ and forming the linear combination of the obtained equalities with coefficients $\lambda_1$, $\dots$, $\lambda_m$ yields the equality $p^C J_p(x)=0$ for some positive integer $C$. As a consequence, we deduce that the group $J_p(\C P^n)$ is a $p$-primary group.
\par
To do this, observe that the coefficients of $1$, $x$, $\dots$, $x^n$ in the linear combination of the right-hand side of \eqref{eq:Jmain} for $q=q_1, \dots, q_m$ and with coefficients $\lambda_1, \dots, \lambda_m$ are the coordinates of the image of the column vector $(\lambda_1, \dots, \lambda_m)^t$ under the linear transformation described by the matrix
$$
M=\left(
\begin{array}{cccc}
1 & 1 & \dots & 1\\
q_1 & q_2 & \dots & q_m \\
{q_1 \choose 2} & {q_2 \choose 2} & \dots & {q_m \choose 2}\\
\vdots & \vdots & \ddots & \vdots \\
{q_1 \choose n} & {q_2 \choose n} & \dots & {q_m \choose n}\\
\end{array}
\right)
$$
If we choose $m=n+1$, then $M$ becomes a square matrix and its determinant can be computed. Indeed, in the $j$th row for $j = 0$, $\dots$, $n$ the $i$th element is the same polynomial of degree $j$ with leading coefficient $1/j!$, evaluated at $q_i$. By induction on $j$, the rows with indices less than $j$ span linearly all the polynomials of the variable $q_i$ of degree less than $j$ for any $i$ and hence adding a suitable linear combination of these rows to the $j$th row we can turn the $j$th row into $\left(\frac{q_i^j}{j!}\right)_{i=1..m}$. Doing this for all $j$, we do not change the determinant and arrive at the matrix $\left(\left( \frac{q_i^j}{j!} \right)\right)_{i=1..m,j=0..n}$, whose determinant is $\prod_{j=0}^{n} \frac{1}{j!}$ times the determinant of the Vandermonde determinant on the elements $q_1, \dots, q_m$, so
$$
\det M = \frac{\prod_{1 \leq v < u \leq m+1} (q_u-q_v) }{\prod_{j=0}^{n} j!}.
$$
\noindent
In particular, if we choose $q_j=pj+1$, then $\det M = p^{n+1 \choose 2}$.

\par
Choosing the (integral) vector $v=(\lambda_1,\dots,\lambda_m)^t$ to be the first column of the cofactor matrix of $M$ (which is $\left(\det M\right) \cdot M^{-1}$), we have $Mv=(\det M,0,\dots,0)^t$, so when we take the linear combination of \eqref{eq:Jmain} for $q=q_i$ with coefficient $\lambda_i$, the right-hand sides sum up to $J_p(\det M \cdot 1)$ and the left-hand sides sum up to $J_p(\sum_i \lambda_i (1+x)) = J_p\left(\left(\det M\right) \cdot (1+x)\right)$. But $J_p(1)=0$ by definition, so we get that $\left(\det M\right) J_p(x) = 0$. Since (with the choice $q_j=pj+1$) the determinant $\det M$ is a power of $p$, $J_p(\C P^n)$ is a $p$-group.

\par

Using the universal property of the $p$-component $J_p^{alg}(\C P^k)$, this proves that the natural projection $J(\C P^k) \to J_p(\C P^k)$ factors through $J_p^{alg}(\C P^k)$. On the other hand, no nonzero element of $J_p^{alg}(\C P^k)$ can vanish in $J_p(\C P^k)$ by \cite[Theorem (1.1)]{AdamsJ}: if $S\xi$ is fibrewise $p$-trivial and hence $\xi$ admits a map of degree $s$ to a trivial bundle, where $s$ is coprime to $p$, then there exists a nonnegative integer $e$ such that $s^e\xi$ is fibrewise homotopy equivalent to a trivial bundle. Multiplication by $s^e$ is an isomorphism in any $p$-group, hence $J(s^e\xi)=0$ and consequently $J_p^{alg}(s^e\xi)=s^eJ_p^{alg}(\xi)=0$ implies that $J_p^{alg}(\xi)=0$.

\end{proof}

\begin{proof}[Proof of periodicity modulo $p$]
Denote by $q$ the quotient $M_k/(M_k)_p$. Then $(p,q)=1$ and $M_k = q \cdot (M_k)_p$. The class of the tautological bundle $\gamma_{k-1}$ in $J(\C P^{k-1})$ has order $M_k$, so its image in $J_p(\C P^{k-1})$ has order $(M_k)_p$. In particular, we have
$$
J_p((M_k)_p\gamma_{k-1})=0.
$$
Consequently for any bundle $\xi$ over $\C P^{k-1}$ we have $J_p(\xi) = J_p(\xi + (M_k)_p \gamma_{k-1})$ and therefore the Thom spaces $T\xi$ and $T(\xi + (M_k)_p \gamma_{k-1})$ are stably homotopically $p$-equivalent. In particular, we obtain a $p$-isomorphism between the groups
$$\pi^{\bf S}_*(\C P^m,\C P^l) \cong \pi^{\bf S}_*(T((l+1)\gamma_{m-l-1}))$$
and
$$\pi^{\bf S}_*(\C P^{m+(M_k)_p},\C P^{l+(M_k)_p}) \cong \pi^{\bf S}_*(T((l+1+(M_k)_p)\gamma_{m-l-1})).$$
Hence (analogously to the proof of Proposition \ref{thm:periodicity}) $E^r_{i,j}$ and $E^r_{i+(M_k)_p,j+(M_k)_p)}$ have canonically isomorphic $p$-components for $r\leq k-1$. This proves the Proposition.
\end{proof}

\section{Image of $J$}

\begin{thm}{\cite[Proposition 4.7(a)]{Mosher}}\label{thm:imJ}
Let $\iota_s \in \pi^{\bf S}_{2s}(\C P^s/\C P^{s-1}) = E^1_{s,s} \cong \Z$ be a generator and suppose that $d^i(\iota_s)=0$ for $i<k$. Then $d^k(\iota_s) \in E^k_{s-k,s+k-1}$ belongs to the image of $\img J \subset E^1_{s-k,s+k-1} = \pi^{\bf S}(2k-1)$ in $E^k_{s-k,s+k-1}$.
\end{thm}

{\bf Remark:} by the geometric interpretation of the singularity spectral sequence (Subsection \ref{section:singSS}) the conclusion of Theorem \ref{thm:imJ} means that under its assumptions, the boundary map of an isolated $\Sigma^{1_{s-1}}$-point can be chosen to be a $\Sigma^{1_{s-k-1}}$-map whose $\Sigma^{1_{s-k-1}}$-set is an immersed framed $(2k-1)$-dimensional sphere. Indeed, $\img J$ consists exactly of those stable homotopy classes which can be represented by a framed sphere.

\begin{proof}
Let us recall the definitions of the differentials $d^i_{s,s}$. For $i=1$ it is the boundary map $\pi^{\bf S}_{2s}(\C P^s, \C P^{s-1}) \to \pi^{\bf S}_{2s-1}(\C P^{s-1}, \C P^{s-2})$. If this is zero on an element $x \in \pi^{\bf S}_{2s}(\C P^s,\C P^{s-1})$ then $x$ comes from $\pi^{\bf S}_{2s}(\C P^s,\C P^{s-2})$, that is, there is an element $x_2 \in \pi^{\bf S}_{2s}(\C P^s,\C P^{s-2})$ such that its image in $\pi^{\bf S}_{2s}(\C P^s,\C P^{s-1})$ is $x$. Then $d^2(x)$ is represented by $\partial(x_2)$, where $\partial$ is the boundary map $\pi^{\bf S}_{2s}(\C P^s, \C P^{s-2}) \to \pi^{\bf S}_{2s-1}(\C P^{s-2},\C P^{s-3})$. Analogously, if $d^{i-1}(x)=0$, then there is an element $x_i \in \pi^{\bf S}_{2s}(\C P^s,\C P^{s-i})$ that is mapped into $x$ by the map $\pi^{\bf S}_{2s}(\C P^s, \C P^{s-i}) \to \pi^{\bf S}_{2s}(\C P^{s},\C P^{s-1})$ and $d^i(x)$ is represented by $\partial (x_i)$, where $\partial$ is the boundary map $\pi^{\bf S}_{2s}(\C P^s, \C P^{s-i}) \to \pi^{\bf S}_{2s-1}(\C P^{s-i},\C P^{s-i-1})$.

\par

Hence the condition $d^i(\iota_s)=0$ for $i<k$ means that in the space $\C P^s/\C P^{s-k-1}$ the top dimensional cell\footnote{We shall consider the natural CW-structure on $\C P^s/\C P^{s-k-1}$ that has one cell in each even dimension $0$, $2(s-k)$, $2(s-k)+2$, $\dots$, $2s$.} is attached to $\C P^{s-1}/\C P^{s-k-1}$ by an attaching map that can be deformed into a map going into the lowest (positive-)dimensional cell of $\C P^{s}/\C P^{s-k-1}$. The class $d^k(\iota_s)$ will be the stable homotopy class of the obtained attaching map $\S^{2s-1} \to \S^{2(s-k)}$. Therefore if we fix the deformation into the lowest cell, then $d^k(\iota_s)$ can be considered as an element of $\pi^{\bf S}(2k-1)$. Without fixing the deformation the value of $d^k(\iota_s)$ will still be well-defined in $E^k_{s-k,s+k-1}$, which is the quotient of $E^1_{s-k,s+k-1} \cong \pi^{\bf S}(2k-1)$ modulo the images of the previous differentials $d^i$, $i<k$. For later reference, we summarize the following lemma:

\begin{lemma}\label{lemma:differential}
Assume that $d^i(\iota_s)=0$ for $i<k$. Then in the space $\C P^s/\C P^{s-k-1}$ the results of different deformations into the bottom cell of the attaching map of the top cell differ by elements of the images of the differentials $d^i$, $i<k$.
\end{lemma}

From this point of view the vanishing of $d^i(\iota_s)$ for $i<k$ means that after contracting the lowest (positive) dimensional cell to a point in $\C P^s/\C P^{s-k-1}$ (that is, after forming the space $\C P^s/\C P^{s-k}$) the top cell splits off stably in the sense that there is a stable map $\S^{2s} \nrightarrow \C P^s/\C P^{s-k}$ such that composing it with the projection $\C P^s/\C P^{s-k} \to \C P^s/\C P^{s-1} = \S^{2s}$ gives a stable map $\S^{2s} \to \S^{2s}$ of degree $1$. In other words, $\C P^s/\C P^{s-k}$ is $S$-reducible.
\par
It is well-known that $\C P^s/\C P^{s-k}$ is the Thom space $T((s-k+1)\gamma_{k-1})$ of the bundle $(s-k+1)\gamma_{k-1}$ over $\C P^{k-1}$. By a result of Atiyah and Adams $T(m\gamma_{k-1})$ is $S$-reducible if and only if $m+k$ is divisible by the number $M_k$ defined in Section \ref{section:periodicity} (see \cite[Theorem 4.3. e)]{Mosher}).

\par
Recall (\cite{Mosher} collecting results of Adams and Walker \cite{AdamsWalker}; James; Atiyah and Todd \cite{AtiyahTodd}) that

\begin{enumerate}
\item the Thom spaces $T(p\gamma_{k-1})$ and $T(q\gamma_{k-1})$ are $S$-dual if $p+q+k \equiv 0 \mod M_k$;
\item $T(n\gamma_{k-1})$ is $S$-coreducible exactly if $n \equiv 0 \mod M_k$; and
\item $T(m\gamma_{k-1})$ is $S$-reducible exactly if $m+k \equiv 0 \mod M_k$.
\end{enumerate}

Let us return to the proof of the theorem. We have seen that $\C P^s/\C P^{s-k} = T\left((s-k+1)\gamma_{k-1}\right)$ is $S$-reducible (hence $s+1 \equiv 0 \mod M_k$ by property $(3)$). Let us take an $n$ such that $\C P^{n+k}/\C P^{n-1} = T\left((n\gamma_{k-1})\right)$ is $S$-dual to $\C P^s/\C P^{s-k-1} = T\left((s-k)\gamma_k\right)$, that is, $n+(s-k)+(k+1) \equiv 0 \mod M_{k+1}$. Assume that $s+1 \equiv tM_k \mod M_{k+1}$ and correspondingly $n \equiv -tM_k \mod M_{k+1}$. We can choose $n$ to be greater than $k$.
\par
Note that the $S$-dual of $\C P^s/\C P^{s-k}$ is $\C P^{n+k-1}/ \C P^{n-1}$ by Lemma \ref{lemma:restriction}. Since $\C P^s/\C P^{s-k}$ is $S$-reducible, its $S$-dual $\C P^{n+k-1}/ \C P^{n-1}$ is $S$-coreducible. Then Lemma \ref{lemma:coreduciblebouquet} together with the condition $n>k$ imply that $\C P^{n+k-1}/ \C P^{n-1}$ is homotopy equivalent to $\S^{2n} \vee (\C P^{n+k-1}/ \C P^n)$.
\par
Since the complex $\C P^{s}/ \C P^{s-k-1}$ becomes $S$-reducible after collapsing the bottom cell, we can consider a $2$-cellification $X=D^{2s} \underset{\alpha}{\cup} \S^{2(s-k)}$ formed by the top and bottom cells. There is a map $f: X \to \C P^s/ \C P^{s-k-1}$ inducing isomorphism in the homologies in dimensions $2s$ and $2(s-k)$. Recall (see the geometric definition of the differentials in the beginning of the proof of Theorem \ref{thm:imJ} and Lemma \ref{lemma:differential}) that here the attaching map $\alpha \in \pi^{\bf S}(2k-1)$ of $X$ is well-defined only up to the choice of the deformation of the attaching map of the top cell, that is, up to an element of the subgroup generated by the images of the lower differentials $d^i$, $i=1,\dots,k-1$. Consider the cofibration
\begin{equation}\tag{*}\label{eq:Xcofibration}
X \to \C P^s/ \C P^{s-k-1} \to \C P^{s-1}/ \C P^{s-k}
\end{equation}
and its $S$-dual cofibration
\begin{equation}\tag{**}\label{eq:Ycofibration}
Y \leftarrow \C P^{n+k}/ \C P^{n-1} \leftarrow \C P^{n+k-1}/ \C P^n.
\end{equation}
By Lemma \ref{lemma:selfDual}, the space $Y$ is also a $2$-cell complex, with the same (stable) attaching map $\alpha$. Note that the cofibration \eqref{eq:Ycofibration} demonstrates the $S$-coreducibility of $\C P^{n+k-1}/\C P^{n-1}$, since after omitting the top cell in both $\C P^{n+k}/\C P^{n-1}$ and in $Y$ the cofibration \eqref{eq:Ycofibration} gives a retraction of $\C P^{n+k-1}/\C P^{n-1}$ onto its bottom cell, $\S^{2n}$. The attaching map of $Y$ is the composition of the attaching map of the top cell of $\C P^{n+k}/\C P^{n-1}$ with this retraction.
\par
Let us denote by $pr: \C P^k \to \C P^k/\C P^{k-1} = \S^{2k}$ the projection. Consider the following diagram (see \cite{Mosher}), where $\varepsilon^d$ denotes the trivial bundle of rank $d$:
\begin{equation}\label{eq:diagramMain}
\begin{split}
\xymatrix{
&\left[n\gamma_k - \varepsilon^n\right] \ar@{}[r]|{\in} \ar[ddl] & \tilde K_{\C}(\C P^k) \ar[d]& \\
& & \tilde K_{\R}(\C P^k) \ar[d]_J & \tilde K_{\R}(\S^{2k}) \ar[l]_{pr^*}\ar[d]_J \\
0 \ar@{}[r]|{\in} & J(\C P^{k-1}) & J(\C P^k) \ar[l] & J(\S^{2k}) \ar[l]_{pr^*}\\
}
\end{split}
\end{equation}
Since $n$ is a multiple of $M_k$, the restriction of the class $\left[n\gamma_k - \varepsilon^n\right] \in \tilde K_{\C}(\C P^k)$ to $\C P^{k-1}$ represents the zero class in $J(\C P^{k-1})$. The bottom row of the diagram \eqref{eq:diagramMain} is exact, so the image of this class in $J(\C P^k)$ (namely $J(n\gamma_k)$) belongs to the image $pr^*(J(\S^{2k}))$. The map $J: \tilde K_{\R}(\S^{2k}) \to J(\S^{2k})$ is surjective by definition, so there exists a (real) vector bundle $\xi$ over $\S^{2k}$ such that the class $[\xi-\varepsilon^{\rk \xi}] \in \tilde K_{\R}(\S^{2k})$ is mapped by $J \circ pr^* = pr^* \circ J$ to $J(n\gamma_k-\varepsilon^n)$. Hence $J(pr^*\xi) = J(n\gamma_k)$. This implies that $pr^*\xi$ and $n\gamma_k$ have stably homotopically equivalent Thom spaces, i.e. (possibly after stabilization to achieve $\rk \xi=2n$) we have $T(pr^*\xi) \cong T(n\gamma_k) = \C P^{n+k}/\C P^{n-1}$.
\par
Note that the space $T(pr^*\xi)$ inherits a cell decomposition from $\C P^k$ in which the top cell corresponds to (the disc bundle of $\xi$ over) the top cell of $\C P^k$. Omitting this top cell gives us the Thom space of the bundle $\left(pr^*\xi\right)|_{\C P^{k-1}} = pr^*\left(\xi|_{\{point\}}\right)$, and this space is coreducible in an obvious way: mapping all the fibers onto the $\S^{2n}$ over the $0$-cell is a retraction onto this $\S^{2n}$. Putting back the top cell and composing its attaching map with the retraction we obtain a $2$-cell complex, which is in fact $T\xi$. The attaching map of $T\xi$ is known to lie in $\img J$ (see eg. \cite{Hatcher}), so it only remains to compare it with the attaching map of $Y$ (which is $d^k(\iota_s)$ modulo the images of the lower differentials). Even though both $Y$ and $T\xi$ are obtained by the same $2$-cellification procedure from stably homotopically equivalent spaces $\C P^{n+k}/\C P^n$ and $Tpr^*\xi$ respectively, the construction depended on the choice of a retraction from the definition of coreducibility of the involved subspaces $\C P^{n+k-1}/\C P^{n-1}$ and $pr^*\xi|_{\C P^{k-1}}$, and there is no reason why the two retractions should coincide.
\par
By Lemma \ref{lemma:dual2cell}, this choice of coreduction of $\C P^{n+k-1}/\C P^{n-1}$ introduces exactly the same ambiguity in the definition of the $2$-cellification of $\C P^{n+k}/\C P^{n-1}$ as the choice of (stable) reducibility of the $S$-dual complex $\C P^s/\C P^{s-k}$, or equivalently the choice of a deformation of the attaching map of the top cell in $\C P^s/\C P^{s-k-1}$ into the bottom cell. By Lemma \ref{lemma:differential}, this choice corresponds exactly to changing our map by an element of the images of the previous differentials $\img d^i$, $i<k$. Hence the attaching map $\alpha$ of the space $X$ lies in the same coset of $\langle \img d^i: i<k \rangle$ as does the attaching map of $T\xi$, which belongs to $\img J$. This finishes the proof. 
\end{proof}
\par
The proof above clarifies the notions behind Mosher's argument not appearing in \cite{Mosher} explicitly, such as $2$-cellification and its dependence on the choice of deformations. Below we give a shorter second proof that hides the geometric machinery within the general framework of spectral sequences.
\begin{proof}
Consider the spectral sequence in stable homotopy associated to the filtration (of length $k+1$)
$$
\C P^{s-k}/\C P^{s-k-1} \subset \C P^{s-k+1}/\C P^{s-k-1} \subset \dots \subset \C P^{s}/\C P^{s-k-1}.
$$
This spectral sequence maps naturally to the spectral sequence of the infinite filtration $\C P^{s-k}/\C P^{s-k-1} \subset \C P^{s-k+1}/\C P^{s-k-1} \subset \dots$ of $\C P^\infty/\C P^{s-k-1}$.
Taking $S$-duals of all the involved spaces and maps, we obtain a sequence of maps that we call a cofiltration:
$$
\C P^{n+k}/\C P^{n+k-1} \leftarrow \C P^{n+k}/\C P^{n+k-2} \leftarrow \dots \leftarrow \C P^{n+k}/\C P^{n-1}
$$
and the spectral sequence $E^r_{i,j}$, $s-k \leq i \leq s$ of the original filtration becomes the spectral sequence $E^{p,q}_r$ of this cofiltration in the stable \emph{cohomotopy} theory, with $E^r_{s-k+p,s-k+q}$ identified with $E_r^{-n-p,-n-q}$ (in particular, $E_1^{p,q}=\pi^{\bf s}(p-q)$ if $-n-k \leq p \leq -n$ and is $0$ otherwise). Explaining more accurately, let $C$ be the middle point of the interval $[-n,s-k]$ on the horizontal axis in the plane. Then reflection in $C$ transforms the groups $E^r_{i,j}$ and the differentials $d^i$ of the original first quadrant spectral sequence into the groups $E^{p,q}_r$ and the differentials $d_i$ of a third quadrant spectral sequence. This latter spectral sequence is that of the $S$-dual cofiltration in stable cohomotopy, because $S$-duality maps stable homotopy groups into stable cohomotopy groups and commutes with homomorphisms induced by mappings. That is, if $A \to X \to X/A$ is a cofibration and $D_S[A] \leftarrow D_S[X] \leftarrow D_S[X/A]$ is the $S$-dual cofibration then the exact sequence in stable homotopy of the first cofibration is canonically isomorphic to the exact sequence in stable cohomotopy sequence of the second:
$$
\xymatrix{
\pi_*^{\bf s}(A) \ar[r] \ar[d]_{\cong}& \pi_*^{\bf s}(X) \ar[r] \ar[d]_{\cong} & \pi_*^{\bf s}(X/A) \ar[d]_{\cong} \\
\pi_*^{\bf s}(D_S[A]) \ar[r]& \pi_*^{\bf s}(D_S[X]) \ar[r] & \pi_*^{\bf s}(D_S[X/A]) \\
}
$$

In this cohomological spectral sequence the differentials $d_1,\dots,d_{k-1}$ going from the column number $-n-k$ all vanish (because $\C P^{n+k-1}/\C P^{n-1}$ admits a retraction onto $\C P^n/\C P^{n-1}$). The last differential $d_k$ maps the class of the identity in $\pi^{\bf S}_{2n+2k}(\C P^{n+k}/\C P^{n+k-1})$ into the attaching map of the top cell. Choosing any (cellular) $2$-cellification map $\C P^{n+k}/\C P^{n-1} \to Z = \S^{2n} \cup D^{2n+2k}$ induces a map of cofiltrations
$$
\xymatrix{
\C P^{n+k}/\C P^{n+k-1} \ar[d] & \C P^{n+k}/\C P^{n+k-2} \ar[l] \ar[d] & \dots \ar[l] & \C P^{n+k}/\C P^{n} \ar[l] \ar[d] & \C P^{n+k}/\C P^{n-1} \ar[l] \ar[d] \\
Z/\S^{2n} & Z/\S^{2n} \ar[l]_{id} & \dots \ar[l]_{id} & Z/\S^{2n} \ar[l]_{id} & Z \ar[l]
}
$$
and hence it induces a map of the corresponding cohomotopical spectral sequences that is an isomorphism on column $-n-k$ in the pages $E^{**}_1$, $E^{**}_2$, $\dots$, $E^{**}_{k-1}$, with ``the same'' differential $d_k$. That is, the attaching map of $Z$ is a representative of the image $d_k\iota^{-n-k,-n-k}$ (where $\iota^{-n-k,-n-k}$ is the positive generator of $E_1^{-n-k,-n-k} \cong \Z$), which coincides with $d^k\iota_{s,s}$, and it lies in $\img J$.
\end{proof} 

Theorem \ref{thm:imJ} has a $p$-localized version that we formulate and prove below. Let us consider the \emph{ $p$-localized Mosher's spectral sequence} ${}^pE^r_{i,j}$, ${}^pd^r$, that is, the spectral sequence defined by the $p$-localization of the usual filtration of $\C P^\infty$. In this spectral sequence the starting groups are the $p$-components of those in Mosher's spectral sequence and the $d^1$ differential is the $p$-component of that in Mosher's spectral sequence.

\begin{thm}\label{thm:imJ_p}
Let $\iota_s \in \pi^{\bf S}_{2s}(\C P^s/\C P^{s-1}) \otimes \Z_{(p)} = {}^pE^1_{s,s} \cong \Z_{(p)}$ be a generator and suppose that ${}^pd^i(\iota_s)=0$ for $i<k$. Then ${}^pd^k(\iota_s) \in {}^pE^k_{s-k,s+k-1}$ belongs to the image of $\img J_p \subset {}^pE^1_{s-k,s+k-1} = \pi^{\bf S}(2k-1) \otimes \Z_{(p)}$ in ${}^pE^k_{s-k,s+k-1}$.
\end{thm}

{\bf Setup:} (the diagram below can help to follow the argument) $\C P^s/\C P^{s-k-1}$ is $S$-dual to $\C P^{n+k}/\C P^{n-1}$ (ensured by $n+s \equiv 0 (M_k)$), and the first nonzero $p$-localized differential from $E^*_{s,s}$ is the $k$-th, that is, $n$ is divisible by $(M_k)_p$, but not by $(M_{k+1})_p$ (this latter statement holds due to Lemma \ref{lemma:JpCP} establishing the equivalence of the geometric and algebraic definitions of $J_p$, which implies that the order of $\gamma_k$ in $J_p(\C P^k)$ is exactly $(M_k)_p$). Then the top cell of $\C P^s/\C P^{s-k-1}$ is attached to the intermediate cells by maps that are trivial after $p$-localization.

\par
$$
\xymatrix{
\C P^s/\C P^{s-k-1}  \ar@{<->}[d]_S& X \ar[l]_-{\deg=q}  \ar@{<->}[d]_S &  A = D^{2s} \underset{\alpha}{\cup} \S^{2(s-k)} \ar[l]  \ar@{<->}[d]_S \\
\C P^{n+k}/\C P^{n-1} \ar[r]^-{/\Z_q} \ar[d]_{\cong_p} & Y \ar[r] & B = D^{2(n+k)} \underset{\alpha}{\cup} \S^{2n} \\ T\left(pr^*\xi\right) \ar[rr] & & T\xi \\
}
$$
The choice of a map $f: D^{2s} \cup \S^{2s-2k} \to \C P^s/\C P^{s-k-1}$ inducing isomorphism in homology in dimensions $2s$ and $2s-2k$ is equivalent to the choice of representative of $\img d_k \iota$ in the coset of the image of all previous differentials going to the same cell (Section \ref{section:SS} and Lemma \ref{lemma:differential}) and is the same as the choice of the map $\S^{2s-1} \to \S^{2s-2k}$ homotopic to the attaching map of the top cell, see Section \ref{section:stable}.
\par
{\it Proof of Theorem \ref{thm:imJ_p}:} We first construct a space $X$ with the same cells as $\C P^s/\C P^{s-k-1}$, where the top cell is actually attached (homotopically nontrivially) only to the bottom cell (not only after $p$-localization), and a map $X \to \C P^s/\C P^{s-k-1}$ that is a $p$-equivalence. Such a space and map exist because the attaching map of the top cell to the intermediate cells is homotopically $p$-trivial, so there is a multiple ($q$-tuple, say) of it that is actually trivial, and adding the last cell with this new attaching map (the attaching map of the top cell multiplied by $q$) we obtain $X$ and a degree $q$ map from $X$ to $\C P^s/\C P^{s-k-1}$ which is the identity when restricted to the $2s-2$-skeleton and a $p$-equivalence when it is restricted to the attached $2s$-cell relative to its boundary, so by the $5$-lemma it is a $p$-equivalence. Denote by $A$ the space formed by the bottom and the top cells of $X$ (a $2$-cellification of $X$). By Lemma \ref{lemma:multiplication}, the $S$-dual of $X$, which we shall denote by $Y$, can be obtained from $D_S[\C P^s/\C P^{s-k-1}] = \C P^{n+k}/\C P^{n-1}$ by wrapping the bottom cell on itself by a degree $q$ map.
\par
By Lemma \ref{lemma:dual2cell}, the possible choices of $A$ (that correspond to different representatives of $\img d_k\iota$ modulo the image of the previous differentials) are $S$-dual to the possible choices of $2$-cellifications of $Y$. In addition, choosing the space $X$ differently does not change the image of (the homotopy class of) the attaching map in the $p$-component (after dividing it by $q$), since for any two approximations there is a common ``refinement'' that factors through both maps $X\to \C P^{n+k}/\C P^{n-1}$. That means that the $2$-cellifications of $Y$ (denoted by $B$ in the diagram) have the attaching map $q \cdot d_k \iota$ modulo the image of the previous differentials and considered in the $p$-component.

\par
We claim that $A$ is $p$-equivalent to the Thom space of a vector bundle over the sphere $\S^{2k}$. To deduce this, we check that over $\C P^{k-1}$ the bundle $n\gamma_{k-1}$ (whose Thom space is $\C P^{n+k-1}/\C P^{n-1}$) is $p$-trivial in the sense that its Thom space is stably $p$-equivalent to the Thom space of the trivial bundle. By equivalence of geometric and algebraic definitions of $J_p$ (Lemma \ref{lemma:JpCP}) it is enough to check that $J(n\gamma_{k-1})$ has order coprime to $p$ and hence vanishes when considered in the $p$-component of $J(\C P^{k-1})$. Indeed, $n$ is divisible by $(M_k)_p$ and consequently $J(n\gamma_{k-1})$ has order coprime to $p$. Hence the $J_p$-image of $n\gamma_k$ is the same as that of the pullback $pr^* \xi$ of a bundle $\xi$ over $\S^{2k}$. From this, we want to conclude that the attaching map in $T\xi$ is ``the same'' as the attaching map in the $2$-cellification $Y$. Both $T(pr^*\xi)$ and $Y$ are coreducible after the removal of the top cell: $Y$ by construction/definition. The space $T(pr^*\xi)$ is coreducible after removal of its top cell because $(pr^*\xi) |_{\C P^{k-1}}$ is trivial and hence $T((pr^*\xi) |_{\C P^{k-1}})$ retracts to $T\varepsilon_{*}^{n} = \S^{2n}$.
\par
{\bf Claim:} Let $U$ and $V$ be $p$-equivalent cell complexes, a single top cell $U_{n+d}$ and $V_{n+d}$, respectively, a single bottom cell $U_{n}$ and $V_{n}$, respectively, and assume that $H_{n+d}(U)$ is generated by $U_{n+d}$, $H_{n}(U)$ is generated by $U_{n}$, $H_{n+d}(V)$ is generated by $V_{n+d}$ and $H_{n}(V)$ is generated by $V_{n}$. Assume that both $U_- = U \setminus U_{n+d}$ and $V_-=V \setminus V_{n+d}$ are coreducible, that is, there exist retractions $\rho_U: U_- \searrow U_{n}$ and $\rho_V: V_- \searrow V_{n}$. Then the $2$-cellifications of $U$ and $V$ are $p$-equivalent modulo the choice of the involved retraction, that is, one can replace the retractions $\rho_U$ and $\rho_V$ by some other retractions $\rho'_U$ and $\rho'_V$ such that the $2$-cell complexes formed from $U$ and $V$ using $\rho'_U$ and $\rho'_V$ are $p$-equivalent.
\par
{\bf Remark:} Applying this Claim to $U=Y$ and $V=T(pr^*\xi)$ we obtain that the $2$-celllifications $B$ and $T\xi$ are $p$-equivalent. In particular, the attaching maps of the $2$-cellifications coincide modulo the images of the lower differentials. For $T\xi$ this attaching map belongs to $\img J$, for $B$ the attaching map is $\alpha$, which is $d^k(\iota_s)$ (modulo lower differentials). Thus Theorem \ref{thm:imJ_p} will be proved as soon as we prove the Claim.
\par
{\it Proof of Claim:} Let $i: U \to V$ be a $p$-equivalence. Without loss of generality, it restricts to a $p$-equivalence of $U_-$ and $V_-$, and maps $U_n$ into $V_n$. If $i|_{U_n}$ were an actual homotopy equivalence onto $V_n$, then we could assume that it is in fact a homeomorphism and we could set $\rho'_U = i^{-1} \circ \rho_V \circ i$. With this choice, $\rho'_U$ is a retraction of $U_-$ onto $U_n$ and is compatible with $i$ in the sense that $i$ descends to a map from the $2$-cellification of $U$ to the $2$-cellification of $V$. This map has the same induced maps in $H_{n}$ and $H_{n+d}$ as $i$ does and is hence a $p$-equivalence.
\par
In general, $i|_{U_n}$ only needs to be a degree $q$ map to $V_n$ for some $q$ coprime to $p$. Let $Cyl_U$ denote the mapping cylinder of $i|_{U_n}$. Define $\hat U$ to be $U$ glued to $Cyl_U$ along $U_n$ (essentially, wrapping the bottom cell of $U$ onto itself by a degree $q$ map). Then $i$ extends naturally to a map $\hat \imath : \hat U \to V$, which is still a $p$-equivalence (checked easily in homology). The retraction $\rho_U$ also extends naturally to $\hat U_- = \hat U\setminus U_{n+d}$ by postcomposing with $i|_{U_n}$. The $2$-cellification of $\hat U$ has an attaching map that is the composition of the attaching map of the $2$-cellification of $U$ with a degree $\deg i|_{U_n}$ map of the bottom cell; this composition does not change the homotopy $p$-type of the glued-together space since we compose with a $p$-equivalence of the bottom cell.
\par
By the first paragraph of the proof the $2$-cellifications of $\hat U$ and $V$ are $p$-equivalent, the $2$-cellifications of $U$ and $\hat U$ are also $p$-equivalent, hence the $2$-cellifications of $U$ and $V$ are $p$-equivalent as well, finishing the proof.


\section{The exact value of the first non-zero differential}

Recall (Subsection \ref{section:singSS}) that vanishing of differentials $d^1$, $\dots$, $d^{k-1}$ on $\iota_s \in E^1_{s,s} = \pi^{\bf S}(0) \cong \Z$ means that taking a germ $\varphi: (\R^{2(s-1)},0) \to (\R^{2s-1},0)$ that has an isolated $\Sigma^{1_{s-1}}$-singularity at the origin, its link map $\partial\varphi:\S^{2s-3} \to \S^{2s-2}$ is cobordant (in the class of $\Sigma^{1_{s-2}}$-maps) to a $\Sigma^{1_{s-k-1}}$-map $\partial_k\varphi$. If $d^k$ is the first differential not vanishing on $\iota_s$, then the image $d^k(\iota_s)$ belongs to the image of the subgroup $\img J \subseteq E^1_{**}$ in $E^k_{**}$ and consequently the map $\partial_k\varphi$ can be chosen in such a way that the top (i.e. $\Sigma^{1_{s-k-1}}$-) singularity stratum of $\partial_k\varphi$ is a sphere $\S^{2k-1}$.

\par
How to determine the exact value of this element $d^k\iota_s$? Mosher \cite{Mosher} answered this question using the so-called $e$-invariants of Adams. Before recalling their precise definitions we collect some properties of the $e$-invariants:
\begin{enumerate}
\item There are homomorphisms $e_R$ and $e_C$ from $\pi^{\bf s}(2k-1) \to \Q/\Z$.
\item The invariant $e_R$ gives a decomposition of $\pi^{\bf S}(2k-1)$ in the sense that $\pi^{\bf S}(2k-1) = \img J \oplus \ker e_R$. In particular, the restriction $e_R|_{\img J}$ is injective.
\item \label{eReC} If $k$ is odd, then $e_C=e_R$. If $k$ is even, then $e_C = 2e_R$.
\item Different representatives of $d^k(\iota_s)$ in $\img J \subseteq \pi^{\bf s}(2k-1)$ may have different values of $e_R$, but the $e_C$-value is the same for all of them. Hence $e_C$ is well-defined on the image of $\img J$, in particular, $e_C(d^k(\iota_s))$ is also well-defined.
\item \label{footnoted} The invariant $e_C$ is injective on the image of $\img J$ in $E^k_{**}$. Hence the value $e_C(d^k(\iota_s))$ determines $d^k(\iota_s)$ uniquely\footnote{It also follows that for any $x \in E^k_{**}$ there are at most $2$ elements in $\img J \subset E^1_{**}$ that are mapped by the partially defined map $E^1_{**} \to E^k_{**}$ onto $x$}.
\end{enumerate}
The precise value of $e_C(d^k(\iota_s))$ is given in the next theorem. Recall that $d^k(\iota_s)$ exists if and only if $s+1$ is divisible by $M_k$, and $d^k(\iota_s)$ vanishes if and only if $s+1$ is divisible by $M_{k+1}$.

\begin{thm}{\textnormal{(}\cite[Proposition 4.7]{Mosher}\textnormal{)}}\label{thm:formulaMosher}
If $s+1 \equiv tM_k$ modulo $M_{k+1}$, then $e_C(d^k\iota_s) = t \cdot u_k$,
where $u_k$ is the coefficient of $z^k$ in the Taylor expansion of
$\left(\frac{\log (1+z)}{z}\right)^{M_k}$.
\end{thm}

Mosher's exposition is rather compressed and hard to understand. We try to
repeat here his argument in a more comprehensible form.

\subsection{The definition of the $e_C$-invariant}

Given $\alpha \in \pi^{\bf S}(2k-1)$, denote by $X_\alpha$ the (stable homotopy
type of the) $2$-cell complex $D^{2q+2k} \underset{f}{\cup} \S^{2q}$, where
$f: \S^{2q+2k-1} \to \S^{2q}$ is a representative of $\alpha$. The Chern
character induces a map of the short exact sequence corresponding to the
cofibration $\S^{2q} \subset X_\alpha \to \S^{2q+2k}$ in the complex $K$-theory into
that in the rational cohomology rings:
$$
\xymatrix{
0 & K(\S^{2q}) \ar[l] \ar[d]_{ch} & K(X_\alpha) \ar[l]^{i^*} \ar[d]_{ch} &
K(\S^{2q+2k}) \ar[l]^{j^*} \ar[d]_{ch} & 0 \ar[l] \\
0 & H^*(\S^{2q}) \ar[l] & H^*(X_\alpha) \ar[l] & H^*(\S^{2q+2k}) \ar[l] & 0 \ar[l]
}
$$
One can choose generators $\zeta_q$, $\zeta_{q+k}$ in $K(X_\alpha)$ as well
as generators $y_q \in H^q(X_\alpha)$ and $y_{q+k} \in H^{q+k}(X_\alpha)$
such that
\begin{align}
ch~ \zeta_{q+k} &= y_{q+k} \\
ch~ \zeta_q &= y_q + \lambda y_{q+k} 
\end{align}
Here $\lambda$ is a rational number that is well-defined up to shifts by
integers. The $e_C$-invariant of $\alpha$ is defined to be
$$
e_C(\alpha) = \lambda \in \Q/\Z.
$$
The definition of the $e_R$ invariant is similar, using real $K$-theory and then the natural complexification functor $K_\R \to K_\C$.

\subsection{The computation of $e_C(d^k(\iota_s))$.}

Since $d^i(\iota_s)=0$ for $i=1,\dots,k-1$, the attaching map of the top
cell of $\C P^s/\C P^{s-k-1}$ can be deformed into the bottom cell. Hence
by collapsing this bottom cell we obtain a reducible space, that is, $\C
P^s/\C P^{s-k}$ is reducible. This latter space is the Thom space
$T(s-k+1)\gamma_{k-1}$. Since $Ta\gamma_{k-1}$ is reducible precisely when
$a+k$ is divisible by $M_k$, we obtain that $s+1 \equiv 0 \mod M_k$.
Similarly $d^k(\iota_s) = 0$ precisely if $s+1 \equiv 0 \mod M_{k+1}$.
Define $t$ to satisfy $s+1 \equiv tM_k \mod M_{k+1}$.

As we have seen before, $T(p\gamma_{k-1})$ and $T(q\gamma_{k-1})$ are $S$-dual if
$p+q+k \equiv 0 \mod M_k$. Hence $\C P^{n+k-1}/\C P^{n-1} = T(n\gamma_{k-1})$
is $S$-dual to $T((s-k+1)\gamma_{k-1}) = \C P^s/\C P^{s-k}$ if $n \equiv 0
\mod M_k$. Analogously, the $S$-dual of $\C P^s/\C P^{s-k-1}$ is $\C
P^{n+k}/\C P^{n-1}$ if $n \equiv -tM_k \mod M_{k+1}$.

We just saw that $\C P^s/\C P^{s-k}$ is reducible, hence its $S$-dual, $\C
P^{n+k-1}/\C P^{n-1}$ is coreducible, that is, it admits a retraction $r:\C
P^{n+k-1}/\C P^{n-1} \to \C P^n/\C P^{n-1}$ onto its bottom cell $\S^{2n} =
\C P^n/\C P^{n-1}$. Then $\C P^{n+k}/\C P^{n-1}$ admits a $2$-cellification
$X = D^{2n+2k} \underset{\alpha}{\cup} \S^{2n}$ and a map
$$
f:\C P^{n+k}/\C P^{n-1} \to X
$$
that coincides with $r$ on $\C P^{n+k-1}/\C P^{n-1}$ and has degree $1$ on the top $(2n+2k)$-cell,
and the attaching map $\alpha$ is a representative of $d^k(\iota_s)$. Our aim is to calculate the $e_C$-invariant of $\alpha$.

Let $p: X \to X/\S^{2n} = \S^{2n+2k}$ be the quotient map, further let $\textbf w$
be the generator of $K(\S^{2n+2k}) \cong \Z$, and denote by $y$ the
generator of $H^2(\C P^\infty)$. Then $y^j$ generates $H^{2j}(\C P^{n+k}/\C
P^{n-1})$. By definition,
\begin{equation}\tag{*}\label{eq:chpw}
ch~ p^*\textbf w = y^n+e_C(\alpha) y^{n+k}.
\end{equation}
Recall that $K(\C P^{n+k}) \cong \Z[\mu]/(\mu^{n+k+1}=0)$, where $\mu$ is
the class of $\gamma_{n+k}-1$. The ring $K(\C P^{n+k}/\C P^{n-1})$ can be
identified with the subring of polynomials divisible by $\mu^n$, that is,
$K(\C P^{n+k}/\C P^{n-1}) \cong \mu^n \cdot \left( \Z[\mu]/(\mu^{k+1}=0)
\right)$. In particular, $p^* \textbf w$ can be written in the form
\begin{equation}\tag{**}\label{eq:pw}
p^*\textbf w = \mu^n \sum_{j=0}^{k} w_j\mu^j,
\end{equation}
where all $w_j$ are integers (depending on $n$ and $k$) and $w_0=1$.

Let us denote $ch~ \mu = e^y-1$ by $z$, then $y=\log (1+z)$. Note that both
$y$ and $z$ can be chosen as generators in the ring of formal power series
$H^{**}(\C P^\infty;\Q)$. Applying $ch$ to both sides of the equality
\eqref{eq:pw} and replacing the left-hand side by the right-hand side of
\eqref{eq:chpw} we obtain
$$
y^n + e_C(\alpha) y^{n+k} = z^n\sum_{j=0}^{k} w_jz^j,
$$
or equivalently
$$
\left( \frac{y}{z} \right)^n = \left(\sum_{j=0}^{k} w_jz^j \right)
(1-e_C(\alpha)y^k+e_C(\alpha)^2y^{2k}-\dots),
$$
Recall that this equality holds whenever $n$ is divisible by $M_k$. Replace $y$ by the corresponding power series in $z$. This is possible because $y=z+\text{higher powers of }z$.
\begin{equation}\tag{***}\label{eq:u_k}
\left( \frac{\log (1+z)}{z} \right)^n = \left(\sum_{j=0}^{k} w_jz^j \right)
(1-e_C(\alpha)z^k+\dots).
\end{equation}
It follows that on the right-hand side of
\eqref{eq:u_k} the first $k$ coefficients of $z^j$ (from $j=0$ to $k-1$) are
integers, and the coefficient of $z^k$ is $-e_C(\alpha)$ modulo $\Z$. By
definition, if $n=M_k$, then the coefficient of $z^k$ on the left-hand side
is $u_k$. If $n$ is divisible by $M_{k+1}$, then the same argument can be
repeated with $k+1$ instead of $k$ and we obtain that the coefficient of
$z^k$ has to be an integer; this means that $e_C(d^k\iota_s) = e_C(\alpha)$ is
$0$ modulo $\Z$. In general, when $n=tM_k$, the left-hand side of
\eqref{eq:u_k} can be expanded as
$$
\left( \frac{\log (1+z)}{z} \right)^n = \left(\left( \frac{\log (1+z)}{z}
\right)^{M_k} \right)^t
$$
using the multinomial theorem. Since the coefficients of $z^j$ in $\left( \frac{\log (1+z)}{z} \right)^{M_k}$ for $j=0,\dots,k-1$ are integers, the same is true for the $t$-th power, and the coefficient of $z^k$ is $tu_k$ modulo $\Z$. On the right-hand side of \eqref{eq:u_k} we see that the coefficient of $z^k$ is $-e_C(\alpha)$ modulo $\Z$, hence $e_C(\alpha) \equiv -tu_k \mod 1$ as claimed. {\hfill\qed}

\par

For example, the first several values of $u_k$ are as follows:
\par
\begin{tabular}{c|c|c|c|c|c|c|c|c|c}
k & 1 & 2 & 3 & 4 & 5 & 6 & 7 & 8 & 9 \\
\cline{1-10}
$M_k$ & 1 & 2 & 24 & 24 & 2880 & 2880 & 362880 & 362880 & 29030400\\
\cline{1-10}
$u_k$ & 1/2 & 11/12 & 0 & 71/120 & 0 & 61/126 & 0 & 17/80 & 0\\
\end{tabular}

\par

{\bf Remark:} Note that the generator $\iota_s$ of the group $\pi^{\bf s}_{2s}(\C P^s, \C P^{s-1}) \cong \Z$ in Mosher's spectral sequence corresponds in the singularity spectral sequence to the cobordism class of a prim map $\sigma_{s-1}: (D^{2s-2},\S^{2s-1}) \to (D^{2s-1},\S^{2s-2})$ that has an isolated $\Sigma^{1_{s-1}}$-point at the origin (given by Morin's normal form of the $\Sigma^{1_{s-1}}$ singularity). Its image $d^k(\iota_s)$ under the first nontrivial differential is represented by the submanifold of $\Sigma^{1_{s-k-1}}$-points of the ``boundary map'' $\partial \sigma_{s-1} : \S^{2s-3} \to \S^{2s-2}$ after eliminating all its higher (than $\Sigma^{1_{s-k-1}}$) singularity strata. The elimination of the higher strata proceeds in several steps and in this process we have to make several choices. First we eliminate the $\Sigma^{1_{s-2}}$ singularities by choosing a cobordism of prim $\Sigma^{1_{s-2}}$-maps that joins $\partial \sigma_{s-1}$ with a $\Sigma^{1_{s-3}}$-map $\partial_1\sigma_{s-1}$. Such a cobordism exists because $d^1(\iota_s)=0$, hence the immersed $\Sigma^{1_{s-2}}$-stratum is null-cobordant, and by \cite{keyfibration} any null-cobordism of this top stratum extends to a cobordism of the entire map $\partial \sigma_{s-1}$. Then we eliminate the $\Sigma^{1_{s-3}}$ singularities by choosing a cobordism of prim $\Sigma^{1_{s-3}}$-maps that joins $\partial_1 \sigma_{s-1}$ with a $\Sigma^{1_{s-4}}$-map, and so on. Finally we obtain a prim $\Sigma^{1_{s-k-1}}$-map. Its $\Sigma^{1_{s-k-1}}$-stratum represents an element $\alpha$ in $\pi^{\bf s}(2k-1)$. For some of these choices the class $\alpha$ (representing $d^k(\iota_s)$) belongs to $\img J$.


\par

As a corollary of the computation of $e_C(d^k(\iota_s))$, we have obtained the surprising fact that whichever representative $\alpha \in \img J \subset E^1_{**} =\pi^{\bf s}(2k-1)$ of the element $d^k(\iota_s)$ we choose, the value $e_C(\alpha)$ is the same. We propose the following explanation of this fact. Consider the following diagram (part of \cite[8.1.]{Mosher}):
$$
\xymatrix{
J(\C P^k) \ar[d]_{\cong}& J(\S^{2k}) \ar[l]_{p_J} \ar[d]_{\theta'_C}\\
J'_C(\C P^k) & J'_C(\S^{2k}) \ar@{>->}[l] \\
}
$$
Here $J'_C(X)$ is a quotient group of $J(X)$ defined in an algebraic way (for the reader's convenience we sketch the definition in the Appendix). Mosher writes: ``elements of $J'_C(\S^{2k})$ are measured by the invariant $e_C$'' and refers to \cite{AdamsJ4}.
This means that $e_C$ is a well-defined and injective map from $J'_C(\S^{2k})$ to $\Q/\Z$. Using this we show that all the representatives of $d^k(\iota_s)$ that belong to $\img J \subset E^1_{**}$ are mapped by $e_C$ into the same element of the group $\Q/\Z$. Indeed, the diagram implies that $\ker p_J = \ker \theta'_C$. The argument of Lemma \ref{lemma:indeterminateJ} shows that $\ker p_J$ is precisely the indeterminancy of the elements of $J(\S^{2k}) \subset E^1_{**}$ in $E^k_{**}$ (that is, the representatives in $\img J \subset E^1_{**}$ of an element of $E^k_{**}$ form a coset of the subgroup $\ker p_J$). Hence all the representatives of $d^k(\iota_s)$ that belong to $\img J=J(\S^{2k})$ will be mapped into the same element in $J'_C(\S^{2k})$ (namely into the unique preimage of $J(n\gamma_k)$ in $J(\C P^k)$) and only elements that represent $d^k(\iota_s)$ will be mapped there.

\par

To complete the explanation we need one more lemma. Recall (see \cite[Chapter 15, Remark 5.3.]{Husemoller}) that there exists a classifying space $BH$ for the semigroup $H$ of degree $1$ self-maps of spheres, and for any $X$ one has
$$
[X,BH]=\tilde K_{top}(X).
$$
Here $\tilde K_{top}(X)$ is the group of stable topological sphere bundles over $X$ up to fiberwise homotopy equivalence. Note that $\tilde K_{top}(\S^r)=\pi^{\bf s}(r-1) = \lim_{q \to \infty} \pi_{q+r-1}(\S^q)$. Furthermore let $n$ be again any sufficiently big natural number with the property that $n+s+1$ is divisible by $M_{k+1}$.

\begin{lemma}\label{lemma:indeterminateJ}
Consider the map $p^*: \tilde K_{top}(\S^{2k}) \to \tilde K_{top}(\C P^k)$ induced by the projection $p: \C P^k \to \C P^k/\C P^{k-1} = \S^{2k}$. Let us consider the sphere bundle $S(n\gamma_k)$ as an element in $\tilde K_{top}(\C P^k)$. Then $(p^*)^{-1}(S(n\gamma_k))$ is precisely the set of elements in $\tilde K_{top}(\S^{2k})=\pi^{\bf s}(2k-1) = E^1_{s-k,s-k+1}$ that represent $d^k(\iota_s) \in E^k_{s-k,s-k+1}$.
\end{lemma}

\begin{proof}
The representatives of $d^k(\iota_s)$ in $E^1_{s-k,s-k+1}$ correspond to the possible choices of deformations of the attaching map of the top cell in $\C P^{n+k}/\C P^{n-1}$ into the bottom cell $\S^{2n}$ (see Lemma \ref{lemma:differential}). We have seen that this is the same as the set of $2$-cellifications of $T(n\gamma_k)$, and this in turn is the same as the choices of a retraction of $\C P^{n+k-1}/\C P^{n-1}= T(n\gamma_{k-1})$ onto the fiber $\S^{2n}$. Such a retraction of $T(n\gamma_{k-1})=S\left(n\gamma_{k-1} \oplus \varepsilon^1\right)/S(\varepsilon^1)$ lifts uniquely to a retraction of the sphere bundle $S\left(n\gamma_{k-1} \oplus \varepsilon^1\right)$, where $\varepsilon^1$ is the trivial real line bundle. This latter retraction can be reinterpreted as a fiberwise homotopy equivalence between the sphere bundles of $n\gamma_{k-1}\oplus\varepsilon^1$ and the trivial bundle $\varepsilon^{2n+1}$; we can therefore consider it to be a topological trivialization (in $\tilde K_{top}(\C P^{k-1})$) of the sphere bundle $S(n\gamma_{k-1})$.
\par
In short, the representatives of $d^k(\iota_s)$ in $E^1_{s-k,s-k+1}$ are in bijection with the topological trivializations of $S(n\gamma_{k-1})$.
\par
Consider the space $\C P^k \cup Cone(\C P^{k-1})$, the two spaces being glued along $\C P^{k-1}$. This space is homotopically equivalent to $\C P^k/\C P^{k-1}=\S^{2k}$ and the inclusion of $\C P^k$ into it is (homotopically) the standard projection of $\C P^k$ onto $\S^{2k}$. Take the element $S(n\gamma_k) \in \tilde K_{top}(\C P^k)$; it corresponds to a homotopy class of maps $\C P^k \to BH$, let $\kappa$ denote one map in this class. The extensions of $\kappa$ to the cone over $\C P^{k-1}$ correspond to topological trivializations of the sphere bundle $S(n\gamma_{k-1})$. On the other hand, these extensions correspond to choosing a preimage of $S(n\gamma_k) \in \tilde K_{top}(\C P^k)$ under the map $p^*$, and we obtain that the choice of representative of $d^k(\iota_s)$ in $E^1_{s-k,s-k+1}$ corresponds to the choice of an element in $(p^*)^{-1}(S(n\gamma_k))$, as claimed.
\end{proof}

{\bf Corollary:} The set of those representatives of $d^k(\iota_s)$ that belong to $\img J = J(\S^{2k})$ is in bijection with the set $(p^*)^{-1}(J(n\gamma_k)) = p_J^{-1}(J(n\gamma_k))$.

\begin{proof}
From Theorem \ref{thm:imJ} we know that $d^k(\iota_s)$ does have representatives in $\img J$ . By restricting the map $p^* :  \tilde K_{top}(\S^{2k}) \to \tilde K_{top}(\C P^k)$ to the elements that can be deformed into $BO \subset BH$, we obtain the map $p_J: J(\S^{2k}) \to J(\C P^k)$.
\end{proof}
In particular, since $\ker p_J$ has size at most $2$, this means that the images of the previous differentials going to $E^*_{s-k,s-k+1}$ intersect $\img J$ in a subgroup of order at most $2$ (see property \eqref{eReC} and footnote to property \eqref{footnoted} of $e$-invariants).

\par

{\bf Remark:} The entire calculation can also be performed for the spectral sequence formed by the $p$-components of the groups of the Mosher spectral sequence. Theorem \ref{thm:imJ_p} proves that the image of the first non-trivial differential from the diagonal of this spectral sequence belongs to the image of $\img J_p$. Indeed, tracing the diagram of Theorem \ref{thm:imJ_p}, if we calculate $e_C(\alpha)$ for the attaching map $\alpha$ of the $2$-cell ``resolutions'' $A$ and $B$, then dividing it back by the degree $q$ we obtain a well-defined value in the $p$-component $\left(\img e_C\right)_p$. The only difference in the actual computation is that the map $\C P^{n+k}/\C P^{n-1} \to B$ is no longer an isomorphism in homology, but induces the multiplication by $q$ on $H^{2n}$. Hence we have to replace $e_C(\alpha)$ throughout the proof with $qe_C(\alpha)$, and in the end we divide the resulting value by $q$ to obtain the $e_C(\alpha)$ in the $p$-component (the division by $q$ may not be meaningful in general, but in the $p$-component it is).
\par
Alternatively, we can choose a $q$ coprime to $p$ such that $d^k(q\iota_s)$ makes sense, then calculate $d^k(\iota_{qs})$ (which makes sense) and divide back by $q$ in the $p$-localization to arrive at $d^k(\iota_s)$ (which only makes sense as ${}^pd^k(\iota_s)$ after localization).

\section{Geometric corollaries}

Here we summarize a few geometric corollaries.

\subsection{Translation of the results to singularities}\label{subsection:geometric}
The first corollary is just a singularity theoretical reinterpretation of the homotopy theoretical results and it answers the following question:
\par
Given non-negative integers $n$, $r_1$ and $r_2$ as well as two stems $\alpha \in \pi^{\bf s}(n-2r_1)$ and $\beta \in \pi^{\bf s}(n-2r_2-1)$, does there exist a prim $\Sigma^{1_{r_1-1}}$-map $f : M^n \to \R^{n+1}$ of a compact $n$-manifold $M$ with boundary $\partial M$ such that $f|_{\partial M}$ is a $\Sigma^{1_{r_2-1}}$-map, the set $\Sigma^{1_{r_1-1}}(f)$ (with its natural framing) represents $\alpha$, while $\Sigma^{1_{r_2-1}}(f|_{\partial M})$ (with the natural framing) represents $\beta$?
\par
Answer: Define $k$ to be the greatest natural number for which $M_k$ divides $r_1+1$ (equivalently, $d^j\iota_{r_1}=0$ for all $j=0$, $\dots$, $k-1$ and $d^k\iota_{r_1} \neq 0$).
\begin{enumerate}[a)]
\item
If $r_2 \geq r_1-k$, then such a map $f$ exists exactly if $\beta - \alpha \cdot d^{r_1-r_2}(\iota_{r_1})$ belongs to the image of the ``lower'' differentials $d^j_{r_2+j,n-r_2-j}$ with $j=1$, $\dots$, $r_1-r_2$.
\item
In general, with $r_1$ and $r_2$ arbitrary, the same condition takes the form $\beta = d^{r_1-r_2} \alpha$ in $E^{r_1-r_2}_{r_1,n-r_1}$ with the additional requirement that the differential has to be defined on $\alpha$. In particular,  when $r_2 < r_1-k$ and we make the additional assumption that $\alpha=\lambda\hat\alpha$ for some $\lambda \in \Z$ for which $d^{r_1-r_2}(\lambda\iota_{r_1})$ is defined, $\beta$ has to be equal to $\hat\alpha \cdot d^{r_1-r_2}(\lambda\iota_{r_1})$ modulo the image of the lower differentials.
\end{enumerate}
\par
For example, if $r_2=r_1-1$, then criterion $a)$ states that such a map $f$ exists exactly if
$$
\beta=\alpha\cdot d^1(\iota_{r_1}) = \begin{cases}
0 & \text{ if } r_1 \text{ is even,}\cr
\alpha\eta & \text{ if } r_1 \text{ is odd,}
\end{cases}
$$
where $\eta \in \pi^{\bf s}(1)$ is the generator.
\par
If $r_2=r_1-2$, we can apply criterion $b)$. When $r_1$ is even, then $\beta$ must lie in the coset $\alpha \cdot d^2(\iota_{r_1}) + \eta \cdot \pi^{\bf s}(n-2r_2-2)$. When $r_1$ is odd and we additionally assume that $\alpha=2\hat\alpha$, then $\beta$ must be $\hat\alpha \cdot d^2(2\iota_{r_1})$.

\subsection{The $p$-localization of the classifying space}

Recall that $X^r = X_{Prim\Sigma^{1_r}}$ denotes the classifying space of prim $\Sigma^{1_r}$-maps of cooriented manifolds. Let $p$ be any prime such that $p \geq r+1$. For any space $Y$ we denote by $(Y)_p$ the $p$-localization of $Y$. Recall that $\Gamma$ stands for $\Omega^\infty S^\infty$.

\begin{thm}
$(X^r)_p \cong \prod_{i=0}^r\left(\Gamma \S^{2i+1} \right)_p$.
\end{thm}

{\bf Remark:} Recall (Section \ref{section:prim}) that $Prim\Sigma^{1_r}(n)$ denotes the cobordism group of prim $\Sigma^{1_r}$-maps of oriented $n$-manifolds into $\R^{n+1}$. Let $\mathcal C_r = \mathcal C(p\leq r+1)$ denote the Serre class of finite abelian groups that have no $p$-primary components for $p>r+1$. We will denote isomorphism modulo $\mathcal C_r$ by $\underset{\mathcal C_r}{\approx}$. Then
$$
Prim\Sigma^{1_r}(n) \underset{\mathcal C_r}{\approx} \oplus_{i=0}^{r} \pi^{\bf s}(n-2i).
$$
This means prim $\Sigma^{1_r}$-maps considered up to cobordism and modulo small primes look as an independent collection of immersed framed manifolds of dimensions $n$, $n-2$, $\dots$, $n-2r$ corresponding to their singular strata.

{\it Proof:} We have seen in Lemma \ref{lemma:main} that $X^r \cong \Omega \Gamma \C P^{r+1}$. Now we use a lemma about the $p$-localization of $\C P^n$:

\begin{lemma} \footnote{We thank D. Crowley for the proof of this lemma.}
For any $p>n+1$ the $p$-localizations of $\C P^n$ and $\S^2 \vee \S^4 \vee \dots \vee \S^{2n}$ are stably homotopically equivalent.
\end{lemma}

\begin{proof}
Serre's theorem states that the stable homotopy groups of spheres $\pi^{\bf s}(m)$ have no $p$-components if $m<2p-3$. Hence after $p$-localization the attaching maps of all the cells in $\C P^n$ will become null-homotopic.
\end{proof}

Hence we get that
$$
(X^r)_p \cong \Omega \Gamma\left( \S^2 \vee \dots \vee \S^{2r+2} \right)_p \cong \Gamma \left( \S^1 \vee \dots \vee \S^{2r+1} \right)_p \cong \prod_{i=0}^r\left(\Gamma \S^{2i+1} \right)_p,
$$
proving the theorem.

\subsection{Odd torsion generators of stable homotopy groups of spheres represented by the strata of isolated singularities}

Example: the odd generator of $\pi^{\bf s}(3) = \Z_{24}$. The isolated cusp map $\sigma_2: (\R^4,0) \to (\R^5,0)$ has on its boundary $\partial \sigma_2: \S^3 \to \S^4$ a framed null-cobordant fold curve. Applying a cobordism of prim fold maps to $\partial \sigma_2$ that eliminates the singularity curve one obtains a map without singularities of a $3$-manifold into $\S^4$. It represents a quadruple of the odd generator in $\pi^{\bf s}(3)$.
\par
Example: the odd torsion of $\pi^{\bf s}(7) = \Z_{15} \oplus \Z_{16}$. Consider an isolated $\Sigma^{1_4}$-map $f:(\R^8,0) \to (\R^9,0)$. Again all the singularity strata of its boundary map $\partial f:\S^7\to \S^8$ can be eliminated (after possibly a multiplication by a power of $2$). The obtained non-singular map of a $7$-manifold into $\S^8$ represents a generator of $\Z_{15}$.
\par
These examples can be produced in any desired amount.

\section{Equidimensional prim maps}

The arguments demonstrated so far can also be applied to the case of codimension $0$ prim maps, both cooriented and not necessarily cooriented. However, the resulting spectral sequences do not have the same richness of structure as Mosher's, so we only indicate the differences from the case of codimension $1$ cooriented prim maps.

\par

The analogue of Lemma \ref{lemma:filtration} that identifies the classifying space of codimension $0$ (not necessarily cooriented) prim maps with $\Omega\Gamma \R P^\infty$ including the natural filtrations goes through without significant changes, giving that the codimension $0$ classifying space $X_{Prim\_\Sigma^{1_r}}(0)$ is $\Omega \Gamma \R P^{r+1}$. Indeed, take a codimension $1$ immersion $f$ with a $\Sigma^{1_r}$ projection. The vertical vector field gives us sections of the relative normal line bundles of the singular strata, and collecting these sections induces the normal bundle of $f$ from the tautological line bundle over $\R P^r$. This defines a map from the classifying space of the lifts of prim $\Sigma^{1_r}$-maps to the Thom space $\R P^{r+1}$ of the tautological line bundle over $\R P^r$, and the $5$-lemma shows that this map is a weak homotopy equivalence. For cooriented maps the corresponding bundle from which the normal bundle of $f$ is induced is the trivial bundle $\varepsilon^1_{\S^r}$ over $\S^r$ and the classifying space $X^{SO}_{Prim\_\Sigma^{1_r}}(0)$ turns out to be $\Omega \Gamma T\varepsilon^1_{\S^r} \cong \Omega \Gamma (\S^{r+1} \vee \S^1)$.
\par
In the cooriented case, the spectral sequence obtained this way has first page $E^1_{p,q} = \pi^{\bf s}(q) \oplus \pi^{\bf s}(0)$ and degenerates to $\pi^{\bf s}(0)$ concentrated in the first column on page $E^2$. In the non-cooriented case, the spectral sequence has first page $E^1_{p,q} = \pi^{\bf s}(q)$ and abuts to $\pi^{\bf s}_*(\R P^\infty)$. The differential $d^1$ can be completely understood: it is induced by the attaching map of the top cell of $\R P^{q+1}$ to the top cell of $\R P^q$, and this attaching map has degree $0$ when $q$ is even and degree $2$ when $q$ is odd. Thus $d^1$ is also $0$ on the even columns and multiplication by $2$ on the odd columns. Consequently the $E^2$ page consists of direct sums of groups $\Z_2$, one for each $2$-primary direct summand in the same $E^1$-cell except at row $0$, where the groups are alternating $0$ and $\Z_2$ (this exception comes from the fact that the group $\pi^{\bf s}(0) = \Z$ is not finite like the rest of the groups $\pi^{\bf s}(q)$).
\par
Periodicity goes through in the same way as before, with $M_k$ replaced by $\vert J(\R P^{k-1}) \vert = 2^{m(k)}$ with $m(k)=\vert\{ 1 \leq p <k : p \equiv 0,1,2,3,4 \text{ mod } 8 \}\vert$ (known from \cite{AdamsJ2}).
\par
Conjecture: the first column of $E^2$ survives to $E^\infty$ without further change, in other words, the differentials $d^2$, $d^3$, $\dots$, ending in the first columns all vanish. This has been observed in the cells number $0$ to $8$ of the first column.

\section*{Appendix: The definition of $J'_C(X)$}

The famous $K$-theoretical $\psi^k$-operations ($k$ is any natural number) of Adams are defined by the properties of
\begin{enumerate}[a)]
\item being group homomorphisms, and
\item satisfying $\psi^k \xi = \xi^k$ if $\xi$ is a line bundle.
\end{enumerate}
Now if $\Phi_K$ is the $K$-theoretical Thom isomorphism, then an operation $\rho^k$ is defined by
$$
\rho^k (\xi) = \Phi^{-1}_K \psi^k \Phi_K (1)
$$
for any vector bundle $\xi$. After having extended $\psi^k$ and $\rho^k$ to virtual bundles one can define the subgroup $V(X) \leq K(X)$ as follows:
$$
V(X) = \left\{ x\in K(X) : \exists y\in K(X) \text{ s.t. } \rho^k(x)=\frac{\psi^k(1+y)}{1+y}\right\}.
$$
Then $J'_C(X) \overset{def}{=} K(X)/V(X)$.

\par

The definition is motivated by the result (proved by Adams) that every $J$-trivial element of $K(X)$ necessarily belongs to $V(X)$. Hence there is a surjection $pr: J(X) \to J'_C(X)$. While the definition of $J(X)$ is geometric (not algebraic) and consequently it is hard to handle, the definition of $J'_C(X)$ is purely algebraic and therefore much easier to compute. Furthermore, the two groups often coincide (eg. for $X=\R P^n$) or are close to each other (for $X=\S^{2k}$ the kernel $\ker pr$ has exponent $2$).

\newpage

\section*{Appendix 2: Calculated spectral sequences}

$E^1_{ij}$ page:
\[
\begin{array}{c|c|c|c|c|c|c|}
\cline{2-7}
 & & & & & & \\
11 & \pi^{\bf s}(10)=\Z_6\langle\eta\circ\mu\rangle\rule{2mm}{0pt} &\hspace*{-14mm}\overset{\eta\circ\mu}{\hbox to12mm{\leftarrowfill}} \ \Z_2^3 &
\hspace*{-9.5mm} \overset{0}{\hbox to12mm{\leftarrowfill}}\ \Z_2 \oplus \Z_2\rule{3.5mm}{0pt} & \hspace*{-7mm}  \overset{\bar \nu + \varepsilon}{\hbox to12mm{\leftarrowfill}}\  \Z_{240}\rule{3mm}{0pt} & \hspace*{-7.5mm} \overset{0}{\hbox to12mm{\leftarrowfill}} \ \Z_2 & 0 \\
  & & & & & & \\
 \cline{2-7}
 & &  & & & & \\
10 & \pi^{\bf s}(9)=\Z_2^3\langle\nu^3,\mu,\eta\circ\varepsilon\rangle  & \hspace*{-4mm} \overset{?}{\hbox to 6mm{\leftarrowfill}}\ \,\Z_2\oplus \Z_2\rule{3mm}{0pt} &
 \hspace*{-11.5mm}  \overset{0}{\hbox to12mm{\leftarrowfill}} \ \Z_{240}
 & \hspace*{-14.5mm} \overset{0}{\hbox to12mm{\leftarrowfill}}\  \Z_2 & 0 & 0 \\
 & & & & & & \\
\cline{2-7}
& & & & & & \\
9 & \pi^{\bf s}(8)=\Z_2\langle\overline\nu\rangle\oplus \Z_2\langle\varepsilon\rangle \rule{3mm}{0pt} & \hspace*{-7.5mm} \overset{\bar \nu + \varepsilon}{\hbox to10mm{\leftarrowfill}} \ \Z_{240}\rule{2mm}{0pt} &
\hspace*{-17.5mm}  \overset{0}{\hbox to12mm{\leftarrowfill}} \  \Z_2 & 0 & 0 & \Z_{24}\\
& & & & & & \\
\cline{2-7}
& & & & & & \\
8 & \pi^{\bf s}(7)=\Z_{240}\langle\sigma\rangle & \hspace*{-15.5mm} \overset{0}{\hbox to12mm{\leftarrowfill}}\ \Z_2 & 0 & 0 &
\Z_{24}\phantom{1} & \hspace*{-7.5mm} \hbox to12mm{\leftarrowfill}\kern-2.5pt\raisebox{.75pt}{$\scriptstyle\langle$} \ \hfil  \Z_2 \\
& & & & & & \\
\cline{2-7}
& & & & & & \\
7 & \pi^{\bf s}(6)=\Z_2\langle \nu^2 \rangle & 0 & 0 & \Z_{24} & \hspace*{-8.5mm}  \overset{0}{\hbox to12mm{\leftarrowfill}}
\  \Z_2 &  \hspace*{-4.5mm}  \overset{\cong}{\hbox to12mm{\leftarrowfill}} \ \Z_2\rule{5mm}{0pt} \\
& & & & & & \\
\cline{2-7}
& & & & & & \\
6 & \pi^{\bf s}(5)=0 & 0 & \Z_{24} & \hspace*{-12mm} \hbox to12mm{\leftarrowfill}\kern-2.5pt\raisebox{.75pt}{$\scriptstyle\langle$}  \ \Z_2\rule{3mm}{0pt} & \hspace*{-7.5mm} \overset{0}{\hbox to12mm{\leftarrowfill}}  \Z_2\rule{3mm}{0pt} & \hspace*{-9mm} \raisebox{.75pt}{$\scriptstyle\langle$}\kern-6pt\hbox to12mm{\leftarrowfill}\  \Z \\
& & & & & & \\
\cline{2-7}
& & & & & & \\
5 & \pi^{\bf s}(4)= 0 &  \Z_{24} & \hspace*{-12mm} \overset{0}{\hbox to12mm{\leftarrowfill}}\ \  \Z_2 &
  \hspace*{-12mm}  \overset{\cong}{\hbox to12mm{\leftarrowfill}}\ \Z_2 &
  \hspace*{-9mm} \overset{0}{\hbox to12mm{\leftarrowfill}}\  \Z & 0\\
 & & & & & & \\
 \cline{2-7}
  & & & & & & \\
4 & \pi^{\bf s}(3)= \Z_{24}\langle\nu\rangle & \hspace*{-14.5mm} \hbox to12mm{\leftarrowfill}\kern-2.5pt\raisebox{.75pt}{$\scriptstyle\langle$}\  \Z_2 &
\hspace*{-15mm} \overset{0}{\hbox to12mm{\leftarrowfill}}\ \ \Z_2 & \hspace*{-15.5mm}  \raisebox{.75pt}{$\scriptstyle\langle$}\kern-6pt\hbox to12mm{\leftarrowfill}\ \  \Z\rule{5mm}{0pt}  & 0 & 0\\
& & & & & & \\
\cline{2-7}
& & & & & & \\
3 & \pi^{\bf s}(2)= \Z_2\langle\eta^2\rangle  & \hspace*{-14.5mm} \overset{\cong}{\hbox to12mm{\leftarrowfill}}\ \  \Z_2 &
\hspace*{-15mm}  \overset{0}{\hbox to12mm{\leftarrowfill}}\ \Z\rule{5mm}{0pt} & 0 & 0 & 0\\
& & & & & & \\
\cline{2-7}
& & & & & & \\
2 & \pi^{\bf s}(1)= \Z_2\langle\eta\rangle  & \hspace*{-15mm}\raisebox{.75pt}{$\scriptstyle\langle$}\kern-6pt\hbox to12mm{\leftarrowfill}\ \ \Z & 0 & 0 & 0 & 0\\
& & & & & & \\
\cline{2-7}
& & & & & & \\
j=1 & \pi^{\bf s}_1(\S^1)= \Z & 0 & 0 & 0 & 0 & 0\\
& & & & & & \\
\cline{2-7}
\multicolumn{1}{c}{\rule{0pt}{20pt}} & \multicolumn{1}{c}{i=0} & \multicolumn{1}{c}{1} & \multicolumn{1}{c}{2} & \multicolumn{1}{c}{3} & \multicolumn{1}{c}{4} & \multicolumn{1}{c}{5}
\end{array}
\]

\newpage
\noindent

$E^2_{ij}$ page:

\vspace*{12pt}
\vbox{$$
\xymatrix{
8 & \Z_2 & \Z_{120} & \Z_2 & 0 & 0 & \Z_{24} \\
7 & \Z_{240} & \Z_2 & 0 & 0 & \Z_{12} & 0 \\
6 & \Z_2 & 0 & 0 & \Z_{24} & 0 & 0 \\
5 & 0 & 0 & \Z_{12} & 0 & 0 & \Z \ar[llu]_{d^2_{5,5}}\\
4 & 0 & \Z_{24} & 0 & 0 & \Z \ar[llu]_{d^2_{4,4}} & 0 \\
3 & \Z_{12}\langle 2\nu \rangle & 0 & 0 & \Z \ar[llu]_{d^2_{3,3}} & 0 & 0\\
2 & 0 & 0 & \Z \ar[llu]_{2 \cdot 2} & 0 & 0 & 0 \\
j=1 & 0 & \Z & 0 & 0 & 0 & 0 \\
& i=0 & 1 & 2 & 3 & 4 & 5 \\
}
$$

\vspace*{-121.0mm}
\[
\hspace*{12mm}\begin{array}{c|c|c|c|c|c|c|}
\cline{2-7}
 & & & & & & \\[.679pt]
 & \hspace*{20.0mm} & \hspace*{15.0mm}& \hspace*{12.5mm} & \hspace*{12.5mm} & \hspace*{12.5mm} & \hspace*{12.5mm} \\[.679pt]
 & & & & & & \\[.679pt]
\cline{2-7}
 & & & & & & \\[.679pt]
 & & & & & & \\[.679pt]
 & & & & & & \\[.679pt]
\cline{2-7}
 & & & & & & \\[.679pt]
 & & & & & & \\[.679pt]
& & & & & & \\[.679pt]
\cline{2-7}
& & & & & & \\[.679pt]
& & & & & & \\[.679pt]
 & & & & & & \\[.679pt]
\cline{2-7}
& & & & & & \\[.679pt]
& & & & & & \\[.679pt]
 & & & & & & \\[.679pt]
\cline{2-7}
 & & & & & & \\[.679pt]
 & & & & & & \\[.679pt]
 & & & & & & \\[.679pt]
\cline{2-7}
 & & & & & & \\[.679pt]
 & & & & & & \\[.679pt]
 & & & & & & \\[.679pt]
\cline{2-7}
 & & & & & & \\[.679pt]
 & & & & & & \\[.679pt]
 & & & & & & \\[.679pt]
\cline{2-7}
\multicolumn{7}{c}{}\\[.679pt]
\multicolumn{7}{c}{}\\[.679pt]
\multicolumn{7}{c}{}\\[.679pt]
\end{array}
\]}

Here, $d^2_{4,4}$ induces the $0$ map between the $3$-components, while the $d^2_{3,3}$ and $d^2_{5,5}$ differentials induce epimorphisms in the $3$-components.

\end{document}